\documentclass[a4paper,10pt]{amsart}

\usepackage[USenglish]{babel}

\usepackage{amsmath}
\usepackage{amssymb}
\usepackage{amsthm}
\usepackage{verbatim}
\usepackage{stackrel}

\usepackage[shortlabels]{enumitem}
\usepackage{comment,cite}	
\usepackage{hyperref}
\hypersetup{
    colorlinks=true,
    linkcolor=blue,
    citecolor=red,
    filecolor=magenta,      
    urlcolor=cyan,
    pdftitle={An operator theoretic approach to uniform (anti-)maximum principles},
    bookmarks=true,
    linktocpage=true,
}

\usepackage[utf8]{inputenc}

\usepackage{amssymb}
\usepackage{amsmath}
\usepackage{amsthm}
\usepackage{bbm}
\usepackage{verbatim, stackrel}

\usepackage[shortlabels]{enumitem}
\usepackage{marginnote}
\usepackage{color}

\usepackage{stmaryrd}
\usepackage{tikz}
\usepackage{tikz-cd}

\newcommand{\bbC}{\mathbb{C}}

\newcommand{\bbN}{\mathbb{N}}

\newcommand{\bbR}{\mathbb{R}}

\newcommand{\bbZ}{\mathbb{Z}}

\newcommand{\calL}{\mathcal{L}}

\DeclareMathOperator{\id}{id} 
\DeclareMathOperator{\one}{{\mathbbm{1}}} 
\DeclareMathOperator{\re}{Re} 
\DeclareMathOperator{\im}{Im} 
\DeclareMathOperator{\dist}{dist} 
\newcommand{\argument}{\mathord{\,\cdot\,}} 
\newcommand{\dx}{\;\mathrm{d}} 
\DeclareMathOperator{\linSpan}{span} 
\newcommand{\norm}[1]{\left\lVert #1 \right\rVert} 
\newcommand{\modulus}[1]{\left\lvert #1 \right\rvert} 
\newcommand{\duality}[2]{\left\langle#1\, ,\, #2\right\rangle} 
\newcommand{\dom}[1]{\operatorname{dom}\left(#1\right)} 
\DeclareMathOperator{\Ima}{Rg} 
\newcommand{\BoundCond}[1]{\operatorname{BC}(#1)}

\newcommand{\spec}{\sigma} 
\newcommand{\resSet}{\rho}
\newcommand{\Res}{\mathcal{R}} 
\newcommand{\pntSpec}{\spec_{\operatorname{pnt}}} 
\newcommand{\spb}{s} 

\newcommand{\Implies}[2]{``\ref{#1} $\Rightarrow$ \ref{#2}'':}

\theoremstyle{definition}
\newtheorem{definition}{Definition}[section]
\newtheorem{remark}[definition]{Remark}

\newtheorem*{remark*}{Remark}
\newtheorem*{remarks*}{Remarks}

\newtheorem{examples}[definition]{Examples}

\theoremstyle{plain}
\newtheorem{proposition}[definition]{Proposition}
\newtheorem{lemma}[definition]{Lemma}
\newtheorem{theorem}[definition]{Theorem}
\newtheorem{corollary}[definition]{Corollary}

\newtheorem{setting}[definition]{Setting}

\numberwithin{equation}{section} 

\begin{document}

\title[Uniform anti-maximum principles]{An operator theoretic approach to uniform (anti-)maximum principles}
\author{Sahiba Arora}
\address{Sahiba Arora, Technische Universität Dresden, Institut für Analysis, Fakultät für Mathematik , 01062 Dresden, Germany}
\email{sahiba.arora@mailbox.tu-dresden.de}
\author{Jochen Gl\"uck}
\address{Jochen Gl\"uck, Universität Passau, Fakultät für Informatik und Mathematik, 94032 Passau, Germany}
\email{jochen.glueck@uni-passau.de}
\date{\today}
\begin{abstract}
	Maximum principles and uniform anti-maximum principles are a ubiquitous topic in PDE theory 
	that is closely tied to the Krein--Rutman theorem and kernel estimates for resolvents.	
	
	We take up a classical idea of Takáč 
	-- to prove (anti-)maximum principles in an abstract operator theoretic framework --
	and combine it with recent ideas from the theory of eventually positive operator semigroups.
	This enables us to derive necessary and sufficient conditions for (anti-)maximum principles 
	in a very general setting.
	Consequently, we are able to either prove or disprove (anti-)maximum principles 
	for a large variety of concrete differential operators.
	As a bonus, for several operators that are already known
	to satisfy or to not satisfy anti-maximum principles,
	our theory gives a very clear and concise explanation of this behaviour.
\end{abstract}

\maketitle

\section{Introduction}
\label{section:introduction}

\subsection*{Maximum and anti-maximum principles}

Maximum principles belong to the most classical tools in the analysis of PDEs. 
If, for instance, $\Omega$ is an open subset of $\bbR^d$ and 
$A: E \supseteq \dom{A} \to E$ is a differential operator 
on a suitable function space $E$ over $\Omega$, 
it is common to say that $A$ satisfies the maximum principle if, for each $f \ge 0$,
a unique solution $u \in \dom{A}$ to the equation
\begin{align*}
	-Au = f
\end{align*}
exists, which in addition, satisfies $u \ge 0$.
For example, if $A$ is the Laplace operator with Dirichlet boundary conditions,
then $A$ satisfies the maximum principle.
If however, $A$ is a second-order operator with some lower order terms,
conditions on the latter are necessary to ensure that the maximum principle holds.
Hence, it actually makes sense to consider the more general equation
\begin{align}
	\label{eq:equation-with-lambda}
	(\lambda - A)u = f
\end{align}
for real numbers $\lambda$ that are located in the resolvent set $\resSet(A)$ of $A$, 
and to ask again whether $f \ge 0$ implies $u \ge 0$.
For second-order elliptic operators $A$ (with local boundary conditions),
this more general version of the maximum principle is now,
as a rule, true whenever $\lambda$ is larger than the so-called \emph{spectral bound}
\begin{align*}
	\spb(A) 
	:=
	\sup 
	\{
		\re \mu: \,
		\mu \in \spec(A)
	\};
\end{align*}
here, $\spec(A)$ denotes the spectrum of $A$.
On the other hand, for elliptic operators of order strictly higher than $2$,
the maximum principle is not satisfied for all $\lambda > \spb(A)$ in general
(this is, for instance, a consequence of the very general results
in \cite[Theorem~3.5]{MiyajimaOkazawa1986} 
and \cite[Proposition~2.2]{ArendtBattyRobinson1990}).
On the other hand, some operators still satisfy the maximum principle 
for $\lambda$ in a right neighbourhood of $\spb(A)$.

In many cases, one will even expect a stronger version of the maximum principle to hold:
by this we mean that, if $f \ge 0$ is non-zero in the equation~\eqref{eq:equation-with-lambda},
then $u$ will, for instance, be strictly positive in $\Omega$, or -- more specifically --
dominate a positive multiple of the leading eigenfunction of $A$.
One (though not the major) consequence of our main result,
Theorem~\ref{thm:main}, is that -- under appropriate assumptions --
the maximum principle already implies its stronger version.

In the setting described above, we can also consider the equation~\eqref{eq:equation-with-lambda}
for $\lambda$ in a left neighboorhood of the spectral bound $\spb(A)$
(or of a more general eigenvalue $\lambda_0$ of $A$).
For such $\lambda$, certain elliptic differential operators $A$ satisfy 
the implication: $f \ge 0$ implies $u \le 0$.
This is typically referred to as an \emph{anti-maximum principle}.
There exists extensive literature about the question
for which operators such an anti-maximum principle holds,
and we refrain from the attempt to give a comprehensive list of references here;
see for instance 
\cite{ClementPeletier1979, Takac1996, ClementSweers2000, ClementSweers2001}
for just a few papers on this.

\subsection*{Uniform vs.\ individual behaviour}

It is crucial to distinguish between the following two versions of the maximum principle
when considering the equation~\eqref{eq:equation-with-lambda}:

Does $f \ge 0$ imply $u \ge 0$ for all $\lambda$ in a right neighbourhood of $\spb(A)$
which might depend on $f$, 
or in a right neighbourhood that can be chosen to be independent of $f$?
In the former case, we speak of an \emph{individual} maximum principle,
while in the latter case, we speak of a \emph{uniform} maximum principle.
The analogous distinction must also be made for the anti-maximum principle.

Individual maximum and anti-maximum principles are, in general, easier to prove
than their uniform versions; 
detailed references on what is known about them in an abstract setting
are given in the subsection on eventually positive resolvents below.
In this paper, we focus on \emph{uniform} (anti-)maximum principles throughout.
Our main result is Theorem~\ref{thm:main} which gives sufficient criteria 
for these principles to hold.
Necessary criteria, on the other hand, can be found 
in Theorem~\ref{thm:domination-interval-one-sided}.

\subsection*{Main notation}

To keep our results at a general and flexible level, 
we will work in the setting of Banach lattices throughout.
This also has the advantage that many theoretical constructions that occur in our arguments
-- in particular, \emph{principal ideals} and norms induced by strictly positive functional,
are most naturally formulated within the theory of Banach lattices.
In the special case of $L^p$-spaces, those constructions are typically framed
as so-called \emph{change of density}-arguments.
Standard references about Banach lattices include the monographs
\cite{Schaefer1974}, \cite{Zaanen1983}, and \cite{Meyer-Nieberg1991}.

Let us briefly recall a bit of notation and terminology that is needed to properly state our setting and our main result.
Fix a complex Banach lattice $E$ with real part $E_\bbR$. 
A linear operator $A: E \supseteq \dom{A} \to E$ is called \emph{real} if its domain satisfies $\dom{A} = \dom{A} \cap E_\bbR + i \dom{A} \cap E_\bbR$ and 
if $A$ maps $\dom{A} \cap E_\bbR$ into $E_\bbR$.
Consequently, a linear operator $T: E \to E$ is real if and only if $T$ maps $E_\bbR$ into $E_\bbR$.

For two vectors $u,v$ in the real part of $E$, we use the notation $v\succeq u$ (or $u\preceq v$) 
if there exists a real number $c \in (0,\infty)$ such that $v \ge c u$; 
equivalently, there exists a real number $\tilde c \in (0,\infty)$ such that $\tilde cv \ge u$. 
Similarly, for two bounded and real linear operators $T,S$ on $E$ 
we use the notation $T \succeq S$ (or $S \preceq T$) 
if there exists a real number $c \in (0,\infty)$ such that $T \ge c S$; 
equivalently, there exists a real number $\tilde c \in (0,\infty)$ such that $\tilde c \, T \ge S$.

For a non-zero vector $u$ in $E$ and a non-zero functional $\varphi \in E'$, we use the notation $u\otimes \varphi$ to denote the rank-$1$ operator on $E$ given by 
\begin{align*}
	(u\otimes \varphi)f = \duality{\varphi}{f} u
\end{align*}
for each $f \in E$.
We say $\varphi$ is strictly positive if $\duality{\varphi}{f}>0$ for all $f\in E_+\setminus \{0\}$.  Moreover, if $u$ is positive, then we use the notation 
\begin{align*}
	E_u := \{x \in E: \; \modulus{x} \preceq u\}
\end{align*}
to denote the so-called \emph{principal ideal} generated by $u$.
It is itself a complex Banach lattice when endowed with the so-called \emph{gauge norm} $\norm{\argument}_u$
given by
\begin{align*}
	\norm{x}_u 
	:=
	\inf 
	\{
		c > 0: \
		\modulus{x} \le c u
	\}
\end{align*}
for each $x \in E_u$.
The vector $u$ is called a \emph{quasi-interior point of $E_+$}
-- or briefly, just a \emph{quasi-interior point} --
if the principal ideal $E_u$ is dense in $E$.
For illustrations of these notions in concrete functions spaces, 
we refer to Examples~\ref{ex:principal-ideal-and-norm-induced-by-functional} below.
A bit of further notation and terminology about Banach lattices can also be found
at the beginning of Section~\ref{section:bounds-of-operators}.

\subsection*{Major assumptions}

Throughout the article, we will often consider the following setting, where we assume a domination or a spectral assumption, or both:

\begin{setting}
	\label{sett:general}
	Let $A: E \supseteq \dom{A} \to E$ be a densely defined, closed, and real linear operator on a complex Banach lattice $E$. 
	Let $u \in E_+$ and let $\varphi \in E'_+$ be strictly positive.
	
	For the vectors $u$ and $\varphi$ we consider the following two assumptions (and we will specify in each result which of the assumptions is used): 
	\begin{enumerate}[\upshape (a)]
		\item 
		\emph{Domination assumption:}
		We assume that $m_1, m_2 \ge 0$ are integers such that the inclusions
		\begin{align*}
			\dom{A^{m_1}} \subseteq E_u \qquad \text{and} \qquad \dom{(A')^{m_2}} \subseteq (E')_{\varphi}
		\end{align*}
		hold; we use the notation $m := m_1 + m_2 \ge 0$.
		
		\item 
		\emph{Spectral assumption:}
		We assume that $\lambda_0 \in \bbR$ is a geometrically simple eigenvalue of $A$ and an eigenvalue of $A'$, with the following additional properties:
		The eigenspace $\ker(\lambda_0-A)$ is spanned by a vector $v$ that satisfies $v \succeq u$, and the dual eigenspace $\ker(\lambda_0-A')$ has an element $\psi$ that satisfies $\psi \succeq \varphi$.
	\end{enumerate}
\end{setting}

Here, as usual, we say an eigenvalue $\lambda_0$ is \emph{geometrically simple} if the dimension of the eigenspace $\ker(\lambda_0-A)$ is equal to one.

It is worthwhile to note that, under the spectral assumption in Setting~\ref{sett:general}, the vector $u$ is automatically a quasi-interior point of $E_+$, since $\dom{A^{m_1}}$ is dense in $E$.
Moreover, we point out that both assumptions in Setting~\ref{sett:general} are invariant 
under simultaneously replacing $A$ with $-A$ and $\lambda_0$ with $-\lambda_0$.

In concrete situations, where $A$ will typically be a differential operator, 
the domination assumption will, in many cases, be a consequence of a Sobolev embedding theorem.
The spectral assumption, on the other hand, has to be checked by different means;
however, the spectral assumption turns out to be, in a sense, necessary for the conclusions 
of our main results to hold; see the explanation at the end of the introduction.
It is also worthwhile noting that our spectral assumption can, in a sense,
be regarded as an abstract version of the spectral conditions that appear in
\cite[Assumption~2 and Theorem~3]{ClementSweers2000} for a class of concrete differential operators.

\subsection*{Main result}

The following theorem is our main result; 
it gives sufficient conditions for the uniform maximum or anti-maximum principle to be satisfied.  
For $\lambda \in \resSet(A)$, we denote the resolvent of $A$ at $\lambda$ 
by $\Res(\lambda,A) := (\lambda - A)^{-1}$, which is a bounded linear operator on $E$.

\begin{theorem}
	\label{thm:main}
	Assume that both the domination and the spectral assumption are satisfied in Setting~\ref{sett:general}. Then every spectral value of $A$ is isolated and the following assertions hold for the eigenvalue $\lambda_0$.
	\begin{enumerate}[\upshape (a)]
		\item\label{thm:main:item:max}
		\emph{Uniform maximum principle} via a lower bound:		
		If there exists a number $\mu_0 > \lambda_0$ in the resolvent set of $A$ 
		such that $\Res(\mu_0, A) \succeq - u \otimes \varphi$, then the estimate
		\[
			\Res(\mu,A) \succeq u \otimes \varphi
		\]
		holds for all $\mu$ in a right neighbourhood of $\lambda_0$.
		
		\item\label{thm:main:item:anti-max}
		\emph{Uniform anti-maximum principle} via an upper bound: 
		If there exists a number $\mu_0 < \lambda_0$ in the resolvent set of $A$ 
		such that $\Res(\mu_0, A) \preceq u \otimes \varphi$, then the estimate
		\[
			\Res(\mu,A) \preceq -u \otimes \varphi
		\]
		holds for all $\mu$ in a left neighbourhood of $\lambda_0$.
	\end{enumerate}
\end{theorem}

After several sections of preparation, 
we will prove this theorem in Section~\ref{section:eventual-positivity}. 
Necessary conditions for uniform (anti-)maximum principles in terms of resolvent estimates
will be given in Theorem~\ref{thm:domination-interval-one-sided},
and by combining these necessary conditions with Theorem~\ref{thm:main}
one can, for instance, find a characterization of the uniform anti-maximum principle
under rather mild assumptions on the operator
(see Theorem~\ref{thm:anti-max-characterization} and Corollary~\ref{cor:anti-max-characterization}).
(Anti-)maximum principles for powers of the resolvent hold under weaker conditions
than in Theorem~\ref{thm:main},
as we show in Theorem~\ref{thm:ev-pos-powers}.
Finally, in Section~\ref{section:applications} we give plenty of applications to concrete differential operators.

In order to provide sufficient context for Theorem~\ref{thm:main}, 
several remarks and explanations are in order.
They are the content of the rest of this introduction.

\subsection*{Strong versions of the (anti-)maximum principle}

An important point to note is that the conclusions in both parts of Theorem~\ref{thm:main}
are not merely that the resolvent satisfies $\Res(\mu, A) \ge 0$ in part~(a) and $\Res(\mu,A) \le 0$ in part~(b),
but in fact, the stronger estimates $\Res(\mu,A) \succeq u \otimes \varphi$ and $\Res(\mu,A) \preceq - u \otimes \varphi$, respectively.
Thus, compared to the discussion at the beginning of the introduction, we  not only prove
a maximum and an anti-maximum principle, but also, in a sense, \emph{strong} versions thereof.

\subsection*{Related literature}

For concrete differential operators, uniform (anti-)maximum principles
have been discussed on many occasions in the literature. 
We refer to the references in Section~\ref{section:applications} for several concrete examples.
As a rule, these results are often based on kernel estimates along
with manipulations that rely on series expansions of the resolvent.

Attempts to unify these arguments in order to 
prove uniform (anti-)maximum principles in a general and abstract setting 
have, on the other hand, been rare.
The most general approach to uniform anti-maximum principles that we are aware of
is due to Tak\'a\v{c} \cite[Section~5]{Takac1996}.
Let us explain in which way our 
Theorem~\ref{thm:main}\ref{thm:main:item:anti-max}
generalizes the sufficient condition for the uniform anti-maximum principle
given by Tak\'a\v{c} in \cite[Theorem~5.2]{Takac1996}:
Assumption~(A2) in \cite[page~347]{Takac1996} requires, in our terminology, 
that the positive cone in $\dom{A}$ -- given by $\dom{A} \cap E_+$ -- has non-empty interior. 
In concrete applications,
this typically comes down to checking the domination assumption $\dom{A^{m_1}} \subseteq E_u$
in our Setting~\ref{sett:general} for the special case $m_1 = 1$.
The fact that we allow a priori for general $m_1$ enables us to 
very freely use Sobolev embedding theorems and in turn, handle a larger class of differential operators.
Assumption~(A3) in \cite[page~347]{Takac1996} requires
that the resolvent of $A$ is \emph{strongly positive}
at all $\lambda > \spb(A)$.
Our spectral assumptions in Setting~\ref{sett:general} are a consequence of this,
but they turn out to be satisfied in many situations where~(A3) is not.
The resolvent estimate~(35) in \cite[Theorem~5.2]{Takac1996}
is related to our setting as follows: 
the upper estimate in~(35) can be seen as a version of our assumption
$\Res(\mu_0, A) \preceq u \otimes \varphi$ in
Theorem~\ref{thm:main}\ref{thm:main:item:anti-max}.
The lower estimate in~(35) is not needed in our setting;
what replaces this estimate is, in a sense, our dual domination assumption
$\dom{(A')^{m_2}} \subseteq (E')_{\varphi}$, which can, again, 
be checked by means of Sobolev embedding theorems in many concrete situations.

Quite abstract results in the spirit of Theorem~\ref{thm:main}
have recently been proved by the first named of the present authors in 
\cite[Theorem~4.1 and Corollary~4.3]{Arora2021};
these results yield a \emph{local} version of the uniform (anti-)maximum principle.
If one chooses $E=F=G$ and $S = T = \id_G$ in \cite[Assumptions~2.1]{Arora2021},
results as in our Theorem~\ref{thm:main} can be recovered from
\cite[Theorem~4.1 and Corollary~4.3]{Arora2021};
yet, assumption~(c) in \cite[Theorem~4.1]{Arora2021} is again the 
domination assumption from our Setting~\ref{sett:general} for
the special case $m_1 = m_2 = 1$.
Thus, our Theorem~\ref{thm:main} can be employed for operators 
defined on larger domains.

\subsection*{Eventual positivity of resolvents}

The terminology about (anti-)maximum principles discussed above can also be rephrased in the language of
\emph{eventually positive} and \emph{eventually negative resolvents} introduced in \cite[Definition~4.1]{DanersGlueckKennedy2016a}. To clarify, fix an isolated real spectral value $\lambda_0$ of $A$. Then saying that $A$ satisfies the uniform maximum principle is the same as saying that $\Res(\argument,A)$ is \emph{uniformly eventually positive}, i.e., the operators $\Res(\lambda,A)$ are positive for all $\lambda$ greater than and sufficiently close to $\lambda_0$. By a positive operator, we mean that the positive cone $E_+$ is invariant under that operator. 

Note that eventually positive resolvent is a generalization of the notion \emph{positive resolvent}, which is defined as: $\Res(\lambda,A)$ exists and is positive for all $\lambda> \lambda_0$. This property is, for example, discussed in detail in \cite{Arendt1987},
and occurs, as mentioned above,
also in Tak\'a\v{c}' paper \cite{Takac1996} as a weaker form of assumption~(A3).
It is also intimately related to \emph{positive $C_0$-semigroups}, see for instance \cite[Corollary~11.4]{BatkaiKramarRhandi2017} (and compare also \cite[Chapters~B-II and~C-II]{Nagel1986}).

The individual maximum principle can similarly be restated by saying that the resolvent
is \emph{individually eventually positive}. Likewise, the anti-maximum principles translate into
uniform and individual \emph{eventual negativity} of the resolvent. In fact, the \emph{strong} (anti-)maximum principle can also be rephrased using the eventual positivity (negativity) terminology; see \cite[Section~4]{DanersGlueckKennedy2016b}. 

Both individual eventual positivity and negativity of the resolvent of various concrete differential operators have occurred on many occasions in the literature (where the term \emph{individual (anti-)maximum principle} was mostly used; see the beginning of the introduction for several references).
Recently, a detailed study of this concept in a general and abstract setting was undertaken in 
\cite[Section~4]{DanersGlueckKennedy2016a}, \cite[Section~4]{DanersGlueckKennedy2016b},
and \cite[Theorem~4.1]{DanersGlueck2017} which was mainly motivated by the analysis of \emph{eventually positive $C_0$-semigroups}.
In relation to this latter notion, the following observation is important:
After various results about and characterizations of
\emph{individually eventually positive} semigroups in the aforementioned papers, even \emph{uniform eventual positivity} was successfully characterized
for many classes of semigroups in \cite[Theorem~10.2.1]{Glueck2016} and \cite{DanersGlueck2018b}. However, uniform eventual positivity of resolvent was, until now, 
not understood at the same general level as for semigroups. 

We briefly note that a simple sufficient criterion for uniform eventual positivity of the resolvent can be derived from the series expansion of the resolvent, as was for instance done in \cite[Propositions~4.2 and 4.3]{DanersGlueckKennedy2016a} and \cite[Proposition~3.2]{DanersGlueck2018a}. For many operators, though, the main assumption of this criterion cannot be checked directly, so the criterion is of rather limited use.
On the other hand, the methods used in \cite[Theorem~10.2.1]{Glueck2016} 
and \cite{DanersGlueck2018b} to obtain uniform eventual positivity of the semigroup are not directly applicable either (this can, for instance, be observed using \cite[Theorem~1]{GrunauSweers2001} with $n=m=k=1$). For these reasons, until now, it has been difficult to obtain satisfactory abstract results about uniform eventual positivity of resolvents. Our Theorem~\ref{thm:main}, along with the necessary conditions 
in Theorem~\ref{thm:domination-interval-one-sided},
closes part of this gap.

\subsection*{Notes on the spectral and the domination condition}

We remark that characterization for individual eventual positivity and negativity of the resolvent was proved under the stronger assumption $\dom{A^{m_1}}\subseteq E_u$ with $m_1=1$ 
in \cite[Theorem~4.4]{DanersGlueckKennedy2016b}. As mentioned before, in Setting~\ref{sett:general}, we weaken this by allowing for general $m_1$. On the other hand, the motivation for the assumption on the dual $A'$ stems from a similar condition imposed on the dual generator in \cite[Theorem~3.1]{DanersGlueck2018b}.

Let us also note that the spectral assumption from Setting~\ref{sett:general} is,
under mild assumptions on the operator $A$, 
necessary for the maximum or anti-maximum principle in the conclusions of Theorem~\ref{thm:main}.
This was proved in \cite[Theorem~4.1]{DanersGlueck2017} and \cite[Corollary~3.3]{DanersGlueckKennedy2016b}.

\section{Upper and lower estimates by rank-$1$ operators}
\label{section:bounds-of-operators}

As indicated by our main result, Theorem~\ref{thm:main}, 
we will throughout deal with upper and lower estimates of bounded linear operators by rank-$1$ operators.
In order to have a well-stocked toolbox available, 
we first prove some fundamental results about those kinds of estimates.

First, we need a few more concepts from Banach lattice theory.
Let $E$ be a real or complex Banach lattice, $u\in E_+$, and $\varphi\in E'$ be strictly positive. 
We have already introduced the principal ideal $E_u$ and the gauge norm on it in the introduction; 
now we explain a construction that behaves, in a sense, dually. 
Define
\[
	\norm{f}_\varphi:=\duality{\varphi}{\modulus{f}} 
\]
for each $f\in E$. 
Then $\norm{\argument}_\varphi$ is a norm on $E$; 
we denote the completion of $E$ with respect to this norm by $E^\varphi$, 
and for the sake of simplicity, we denote the norm on $E^\varphi$ also by $\norm{\argument}_\varphi$. 
It turns out that $E^\varphi$ is a Banach lattice 
with the special property that $\norm{f+g}_\varphi = \norm{f}_\varphi+ \norm{g}_\varphi$ 
for all positive elements $f,g\in E^\varphi$, 
i.e., $E^\varphi$ is a so-called \emph{AL-space} 
(and thus, it is isometrically lattice isomorphic to an $L^1$-space \cite[Theorem~2.7.1]{Meyer-Nieberg1991}).

Both mappings $E_u \to E$ and $E \to E^\varphi$ are continuous, injective, and lattice homomorphisms. 
Of course, the second embedding always has a dense range 
and the range of the first embedding is dense if and only if $u$ is a quasi-interior point of $E$.
A bit more information on the space $E^\varphi$ and its relation to $E_u$ can be found 
at the beginning of \cite[Section~IV.3]{Schaefer1974}.

Let us illustrate the abstract construction outlined above, 
along with the concept of principal ideals and the gauge norm,
by two simple concrete examples.

\begin{examples}
	\label{ex:principal-ideal-and-norm-induced-by-functional}	
	\begin{enumerate}[(a)]
		\item 
		Let $(\Omega,\mu)$ be a finite measure space and $p \in [1,\infty]$. 
		Let $\one$ denote the constant one function in the space $E:=L^p(\Omega,\mu)$, 
		and consider the functional $E' \ni \varphi: f \mapsto \int_\Omega f\, d\mu$.
	
		Then it is not difficult to check that $L^\infty(\Omega,\mu) = E_{\one}$
		and that the gauge norm with respect to $\one$ coincides with the sup norm. 
		Moreover, the space $E^\varphi$ is $L^1(\Omega,\mu)$, and the norm $\norm{\argument}_\varphi$ 
		is simply the $1$-norm.
		
		\item
		Let $K$ be a compact Hausdorff space and consider the Banach lattice $E = C(K)$,
		i.e., the space of continuous scalar-valued functions on $K$, endowed with the sup norm.
		
		Let $u \in C(K)$ be a function which is strictly positive everywhere on $K$. 
		Then $C(K)_u = C(K)$, 
		and the gauge norm $\norm{\argument}_u$ is equivalent to the sup norm on $C(K)$; 
		both norms coincide if and only if $u$ is the constant function with value $1$.
		
		If $\mu \in C(K)'$ is a (Radon) measure on $K$ which is strictly positive on every non-empty open subset of $K$,
		then $C(K)^\mu$ is isometrically lattice isomorphic to the space $L^1(K, \mu)$.
	\end{enumerate}
\end{examples}

We now use the spaces $E_u$ and $E^\varphi$ to characterize 
a joint upper and lower estimate of a bounded linear operator by a rank-$1$ operator.
In part~\ref{prop:char-of-rank-1-domination:item:extension}
of the following proposition, $E_u$ is endowed with the gauge norm $\norm{\argument}_u$.

\begin{proposition}
	\label{prop:char-of-rank-1-domination}
	Let $E$ be a real or complex Banach lattice, let $u \in E_+$, and let $\varphi \in E'_+$ be strictly positive. For every bounded real linear operator $T: E \to E$, the following assertions are equivalent:
	\begin{enumerate}[ref=(\roman*)]
		\item\label{prop:char-of-rank-1-domination:item:modulus-domination} 
		There exists a number $c \in [0,\infty)$ such that $\modulus{Tx} \le c \, \duality{\varphi}{\modulus{x}} u$ for each $x \in E$.
		
		\item\label{prop:char-of-rank-1-domination:item:both-sides-domination} 
		There exists a number $\hat c \in [0,\infty)$ such that 
		$- \hat c \, u \otimes \varphi \le T \le \hat c \, u \otimes \varphi$
		(in short: we have $- u \otimes \varphi \preceq T \preceq u \otimes \varphi$).
		
		\item\label{prop:char-of-rank-1-domination:item:extension} 
		There exists a bounded linear operator $\tilde T: E^\varphi \to E_u$ 
		such that the following diagram commutes:
		\begin{center}
			\begin{tikzcd}
				E^\varphi      \arrow{rr}{\tilde T} & & E_u \arrow{d}{k} \\
				E \arrow{u}{j} \arrow{rr}{T}        & & E 
			\end{tikzcd}
		\end{center}
		Here, $j: E \to E^\varphi$ and $k: E_u \to E$ are the canonical embeddings.
	\end{enumerate}
\end{proposition}

\begin{proof}
	\Implies{prop:char-of-rank-1-domination:item:modulus-domination}{prop:char-of-rank-1-domination:item:both-sides-domination}
	The inequality in~\ref{prop:char-of-rank-1-domination:item:modulus-domination} implies that
	\[
		-c \, \duality{\varphi}{x} u 
		\le
		- \modulus{Tx}
		\le 
		Tx
		\le 
		\modulus{Tx}
		\le 
		c \, \duality{\varphi}{x} u
	\]
	for each $x \in E_+$, which yields~\ref{prop:char-of-rank-1-domination:item:both-sides-domination}
	for $\hat c = c$.
	
	\Implies{prop:char-of-rank-1-domination:item:both-sides-domination}{prop:char-of-rank-1-domination:item:modulus-domination}
	By~\ref{prop:char-of-rank-1-domination:item:both-sides-domination},
	there exists a number $\hat c \ge 0$ such that
	$-\hat  c \, u \otimes \varphi \le T \le\hat  c \, u \otimes \varphi$. 
	This implies that, for each $x \in E_+$, we have $\modulus{Tx} \le\hat  c \duality{\varphi}{x} u$.
	For each vector $x$ in the real part of $E$, we thus obtain
	\[
		\modulus{T x^+} \le\hat  c \duality{\varphi}{x^+} u
		\qquad \text{and} \qquad 
		\modulus{T x^-} \le\hat  c \duality{\varphi}{x^-} u
	\]
	and hence, $\modulus{Tx} \le \modulus{T x^+} + \modulus{T x^-} \le\hat  c \duality{\varphi}{\modulus{x}} u$.
	
	Finally, in case that scalar field is complex, consider a general vector $x \in E$.
	The modulus of each vector $z \in E$ can be written as
	\begin{align*}
		\modulus{z} 
		= 
		\sup_{\theta \in [0,2\pi)} \modulus{(\cos \theta) \re z + (\sin \theta) \im z}
	\end{align*}
	and, alternatively, as
	\begin{align*}
		\modulus{z} 
		= 
		\sup_{\theta \in [0,2\pi)} \modulus{\cos \theta} \modulus{\re z}  +  \modulus{\sin \theta} \modulus{\im z};
	\end{align*}
	the first formula is the definition of $\modulus{z}$ in a complex Banach lattice, 
	and the second formula follows,
	e.g., from Kakutani's representation theorem for AM-spaces \cite[Theorem~2.1.3]{Meyer-Nieberg1991}
	(applied to the principal ideal generated by $\modulus{\re z} + \modulus{\im z}$)
	since the formula holds for complex numbers $z$.
	
	If we apply the first formula to $\modulus{Tx}$ and the second formula to $\modulus{x}$, 
	inequality~\ref{prop:char-of-rank-1-domination:item:modulus-domination} follows 
	since we have already proved it for real vectors $x$.
	
	\Implies{prop:char-of-rank-1-domination:item:modulus-domination}{prop:char-of-rank-1-domination:item:extension} 
	If $\modulus{Tx} \le c \duality{\varphi}{\modulus{x}} u$ for each $x \in E$, 
	then $T E \subseteq E_u$ and ${\norm{Tx}_u \leq c \norm{x}_{E^{\varphi}}}$ for each $x\in E$. 
	Using the density of $E$ in $E^\varphi$, the existence of $\tilde T$ follows.
	
	\Implies{prop:char-of-rank-1-domination:item:extension}{prop:char-of-rank-1-domination:item:modulus-domination} 
	By the commutativity of the diagram, $T$ maps into $E_u$ 
	(more precisely, into the image of $E_u$ under $k$). Let $c:=\norm{\tilde T}_{E^\varphi \to E_u}$.
	For each $x \in E$ we have
	\[
		\norm{\tilde T j x}_u \le c \norm{j x}_\varphi = c\duality{\varphi}{\modulus{x}},
	\]
	so $\modulus{\tilde T j x} \le c\duality{\varphi}{\modulus{x}} u$.
	From this, we conclude that
	\[
		\modulus{Tx} = c\modulus{k \tilde T j x} \le c\duality{\varphi}{\modulus{x}} u,
	\]
	which proves~\ref{prop:char-of-rank-1-domination:item:modulus-domination}.
\end{proof}

The following simple observation about the constants in the previous proposition
will be useful later on:

\begin{remark}
	\label{rem:char-of-rank-1-domination-constants}
	The proof of Proposition~\ref{prop:char-of-rank-1-domination} shows that:
	if~\ref{prop:char-of-rank-1-domination:item:extension} holds,
	then~\ref{prop:char-of-rank-1-domination:item:modulus-domination}
	and~\ref{prop:char-of-rank-1-domination:item:both-sides-domination}
	hold with
	\[
		\hat c = c = \norm{\tilde T}_{E^\varphi \to E_u}.
	\]
\end{remark}

A nice consequence of the previous proposition is the following multiplicative permanence property 
of operators that can be estimated above and below by a given rank-$1$ operator.

\begin{proposition}
	\label{prop:rank-1-domination-composition}
	Let $E$ be a real or complex Banach lattice, let $u \in E_+$, 
	and let $\varphi \in E'_+$ be strictly positive.
	Let $T_1, S, T_2: E \to E$ be bounded real linear operators 
	and assume that $T_1$ and $T_2$ both satisfy
	the equivalent conditions of Proposition~\ref{prop:char-of-rank-1-domination}.
	
	Then $T_2 S T_1$ also satisfies 
	the equivalent conditions of Proposition~\ref{prop:char-of-rank-1-domination}.
\end{proposition}

\begin{proof}
	We use the the same notation as in Proposition~\ref{prop:char-of-rank-1-domination}.
	From the commutative diagram
	\begin{center}
		\begin{tikzcd}
			E^\varphi      \arrow{rr}{\tilde T_1} & & E_u \arrow{d}{k}  & & E^\varphi      \arrow{rr}{\tilde T_2} & & E_u \arrow{d}{k}
			\\
			E \arrow{u}{j} \arrow{rr}{T_1}        & & E   \arrow{rr}{S} & & E \arrow{u}{j} \arrow{rr}{T_2}        & & E
		\end{tikzcd}
	\end{center}
	we immediately get the commutative diagram
	\begin{center}
		\begin{tikzcd}
			E^\varphi      \arrow{rrr}{\tilde T_2 j \, S \, k \tilde T_1} & & & E_u \arrow{d}{k} \\
			E \arrow{u}{j} \arrow{rrr}{T_2 S T_1}               & & & E
		\end{tikzcd}
		\, ,
	\end{center}
	which proves the assertion.
\end{proof}

Let us now explain how the equivalent conditions in Proposition~\ref{prop:char-of-rank-1-domination}
can be checked in concrete situations.
Here, we need a result from \cite[Proposition~2.1]{DanersGlueck2018b} which says the following:
let $E$ be a complex Banach lattice and let $\varphi \in E'$ be strictly positive.
If $T: E \to E$ is a bounded linear operator and 
the range of its dual operator satisfies the domination condition
\begin{align*}
	T'E' \subseteq (E')_\varphi,
\end{align*}
then $T$ extends to a bounded linear operator $E^\varphi \to E$.
This result easily implies the following:

\begin{proposition}
	\label{prop:domination-imlies-estimate-for-product}
	Let $E$ be a real or complex Banach lattice, let $u \in E_+$, 
	and let $\varphi \in E'_+$ be strictly positive.
	If $T_1,T_2,T_3: E \to E$ are bounded real linear operators which satisfy
	the domination conditions
	\begin{align*}
		T_1'E' \subseteq (E')_\varphi
		\qquad \text{and} \qquad 
		T_3 E \subseteq E_u,
	\end{align*}
	then the product $T := T_3 T_2 T_1$ satisifies the equivalent 
	conditions~\ref{prop:char-of-rank-1-domination:item:modulus-domination}--\ref{prop:char-of-rank-1-domination:item:extension}
	in Proposition~\ref{prop:char-of-rank-1-domination}.
\end{proposition}
\begin{proof}
	We prove that the operator $T = T_3 T_2 T_1$
	satisfies condition~\ref{prop:char-of-rank-1-domination:item:extension}
	in Proposition~\ref{prop:char-of-rank-1-domination}.
	
	To this end, first note that, by the closed graph theorem, 
	$T_3$ acts continuously as an operator from $E$ to $E_u$ 
	(where the latter space is endowed with the gauge norm $\norm{\argument}_u$).
	Strictly speaking, an operator is determined not only by its domain and action,
	but also by its range -- so let us be precise here and 
	use the notation $\hat T_3: E \to E_u$ for the operator
	from $E$ to $E_u$ that acts as $T_3$.
	
	As mentioned right before the proposition,
	the operator $T_1$ extends to a continuous linear operator $\hat T_1$ 
	from $E^\varphi$ to $E$.
	Hence, condition~\ref{prop:char-of-rank-1-domination:item:extension}
	in Proposition~\ref{prop:char-of-rank-1-domination}
	is satisfied for the operator $\tilde T := \hat T_3 T_2 \hat T_1: E^\varphi \to E_u$.
\end{proof}

As an immediate consequence of the previous proposition, 
we get a resolvent estimate if our domination assumption from Setting~\ref{sett:general}
is satisfied:

\begin{corollary}
	\label{cor:two-sided-estimate-of-resolent-power-by-domination-assumption}
	Assume that the domination assumption in Setting~\ref{sett:general} is satisfied
	and let $\mu,\mu_1$, and $\mu_2$ be real numbers in the resolvent set $\resSet(A)$.
	
	Then the operator $\Res(\mu_1,A)^{m_1} \Res(\mu_2,A)^{m_2}$ satisfies the equivalent assertions
	in Proposition~\ref{prop:char-of-rank-1-domination};
	in particular, the power $\Res(\mu,A)^m$ satisfies these equivalent assertions as well.
\end{corollary}

Let us close this section with a note on compactness 
for operators which satisfy the equivalent conditions of 
Proposition~\ref{prop:char-of-rank-1-domination}. 
Recall that a real Banach lattice $E$ is said to be \emph{order complete} 
if every non-empty order bounded subset of $E$ has a supremum in $E$
and a complex Banach lattice $E$ is called \emph{order complete} 
whenever its real part is order complete. 
Moreover, a linear operator between two Banach lattices is called \emph{regular} 
if it can be written as a difference of two positive operators.

\begin{corollary}
	\label{cor:dom-implies-compactness}
	In the situation of Proposition~\ref{prop:char-of-rank-1-domination}, 
	assume that the bounded real linear operator $T: E \to E$ satisfies the equivalent 
	assertions~\ref{prop:char-of-rank-1-domination:item:modulus-domination}--\ref{prop:char-of-rank-1-domination:item:extension}. 
	Then $T^3$ is compact.
\end{corollary}
\begin{proof}
	We may assume throughout the proof that the underlying scalar field is real. 
	By assumption, there exists a number $c \in (0,\infty)$ such that
	$- c \; u \otimes \varphi \le T \le c \; u \otimes \varphi$.
	For the following argument, we need that the operator under consideration possesses a positive and a negative part; 
	since this is, in general, only true on order complete Banach lattices, therefore, let us switch to the bi-dual space of $E$:
	
	We have $(u \otimes \varphi)'' = u \otimes \varphi$, where we interpret $\varphi$ as a functional on $E''$. 
	Clearly, the estimate 
	\begin{align*}
		- c \; u \otimes \varphi \le T'' \le c \; u \otimes \varphi
	\end{align*}
	holds. This inequality shows that $T''$ is a regular operator, 
	and hence -- as $E''$ is order complete -- $T''$ has a modulus in $\calL(E'')$ 
	(see e.g.\ \cite[Proposition~IV.1.2]{Schaefer1974}) 
	which satisfies $\modulus{T''} \le c \; u \otimes \varphi$. 
	Consequently, the positive and the negative part of $T''$ 
	-- i.e., $(T'')^+$ and $(T'')^-$ -- 
	exist and are also dominated by the rank-$1$ operator $c \; u \otimes \varphi$.
	
	It is known that, if three positive operators on a Banach lattice are dominated by compact operators, 
	then their product is compact as well \cite[Corollary~3.7.14]{Meyer-Nieberg1991}; 
	hence, the operator
	\begin{align*}
		(T^3)'' = (T'')^3 = \left( (T'')^+ - (T'')^- \right)^3
	\end{align*}
	is the sum of eight compact operators, and is thus itself compact. 
	Consequently, $T^3$ is compact, as asserted.
\end{proof}

We note that the third power in Corollary~\ref{cor:dom-implies-compactness} is optimal in general:
for instance, the positive operator $S$ in \cite[Example~(ii) on p.\,224]{Meyer-Nieberg1991}
is dominated by a positive rank-$2$ operator -- and hence, also by a positive rank-$1$ operator --
but its square $S^2$ is not compact.

By combining Proposition~\ref{prop:domination-imlies-estimate-for-product}
and Corollary~\ref{cor:dom-implies-compactness} one can conclude that,
if a bounded real operator $T: E \to E$ satisfies the domination conditions 
\begin{align*}
	T'E' \subseteq (E')_\varphi
	\qquad \text{and} \qquad 
	T E \subseteq E_u,
\end{align*}
for a vector $0 \le u \in E$ and a strictly positive functional $\varphi \in E'$,
then the power $T^6$ is compact.
However, this result is not optimal under the given domination assumptions:
in \cite[Corollary~2.3]{DanersGlueck2018b} it was shown that 
even $T^4: E \to E$ is compact.

\section{Compactness of the resolvent}
\label{section:resolvent-compactness}

In both \cite[Section~4]{DanersGlueckKennedy2016a} and  \cite[Section~4]{DanersGlueckKennedy2016b}, 
\emph{individual} eventual strong positivity of the resolvent of an operator $A$ at a pole was characterized. 
In particular, it was shown that the pole must be algebraically simple 
and have its order equal to $1$. 
In the following proposition, we show that the domination assumption in Setting~\ref{sett:general} 
guarantees the eigenvalue $\lambda_0$ is a pole of the resolvent $\Res(\argument,A)$, 
and that, in conjunction with the spectral assumption,
we even get that the eigenvalue $\lambda_0$ is algebraically simple with pole order $1$. Recall that an eigenvalue $\lambda_0$ is said to be \emph{algebraically simple} if the dimension of the generalized eigenspace $\bigcup_{n\in\bbN} \ker(\lambda_0-A)^n$ is equal to one.

\begin{proposition}
	\label{prop:spectral-properties}
	Assume that the domination assumption is satisfied in Setting~\ref{sett:general}. Then:
	\begin{enumerate}
		\item[\upshape (a)] 
		If $\mu$ is a real number in the resolvent set $\resSet(A)$,
		then the operator $\Res(\mu,A)^{3m}$ is compact and thus, 
		$A$ has at most countably many spectral values, 
		and each spectral value is isolated and a pole of the resolvent with finite-rank residuum.
	\end{enumerate}
	If, in addition, the spectral assumption is also satisfied, then:
	\begin{enumerate}
		\item[\upshape (b)] 
		The order of $\lambda_0$ as a pole 
		of the resolvents $\Res(\argument,A)$ and $\Res(\argument,A')$ equals $1$;
		hence, $\lambda_0$ is an algebraically simple eigenvalue of both $A$ and $A'$.
		
		\item[\upshape (c)] 
		The eigenvectors $v \in \ker(\lambda_0-A)$ and $\psi \in \ker(\lambda_0-A')$ 
		satisfy $\duality{\psi}{v} > 0$, 
		and the operator $\frac{\psi \otimes v}{\duality{\psi}{v}}$ is the spectral projection of $A$ 
		associated with the spectral value $\lambda_0$.
	\end{enumerate}
\end{proposition}

\begin{proof}
	(a) 
	Fix $\mu \in \resSet(A)$. 
	If $\mu$ is real, the operator $\Res(\mu,A)^m$ is real, 
	so the compactness of $\Res(\mu,A)^{3m}$ is an immediate consequence 
	of Corollary~\ref{cor:two-sided-estimate-of-resolent-power-by-domination-assumption}
	and the compactness assertion in Proposition~\ref{cor:dom-implies-compactness}.
	The rest of the assertions are standard results of the Riesz--Schauder theory of compact operators (see \cite[Section~X.5]{Yosida1980} and \cite[Corollary~IV.1.19]{EngelNagel2000}).
	
	For the rest of the proof, we assume that the spectral assumption is satisfied as well. 
	
	(b) 
	The vector $v$ is a quasi-interior point since $v \succeq u$, 
	and the functional $\psi$ is strictly positive since $\psi \succeq \varphi$. 
	Due to these observations, the fact that the order of $\lambda_0$ as a pole of 
	$\Res(\argument,A)$ and $\Res(\argument,A')$ equals $1$ 
	can be found in~\cite[Proposition~3.1]{DanersGlueckKennedy2016b}.
	Since $\lambda_0$ is -- by the spectral assumption -- a geometrically simple eigenvalue of $A$,
	the result that the pole order is $1$ also implies that $\lambda_0$
	is an algebraically simple eigenvalue of both $A$ and $A'$.
	
	(c) 
	Due to the strict positivity of $\psi$ we have $\duality{\psi}{v} > 0$. 
	As $\lambda_0$ is a pole of the resolvent $\Res(\argument,A)$ of order equal to $1$, 
	the corresponding spectral projection $P$ satisfies $\Ima P=\ker(\lambda_0-A) = \linSpan\{v\}$ 
	and $\Ima P'=\ker(\lambda_0-A') = \linSpan\{\psi\}$. 
	This implies that $P=\frac{\psi \otimes v}{\duality{\psi}{v}}$.
\end{proof}

In many concrete situations, $A$ will be the generator of a $C_0$-semigroup 
which implies that $\lambda \Res(\lambda,A)$ is bounded as $\lambda \to \infty$. 
This latter property suffices to conclude that the resolvent is itself compact
if one of its powers is compact. 
In light of Proposition~\ref{prop:spectral-properties}(a), 
we find it worthwhile to include this fact here:

\begin{proposition}
	\label{prop:compact-resolvent-by-compact-power}
	Let $A: E \supseteq \dom{A} \to E$ be a bounded linear operator on a complex Banach space $E$. 
	Assume that all sufficiently large real numbers $\lambda$ 
	are contained in the resolvent set $\resSet(A)$ of $A$ and 
	that $\lambda \Res(\lambda,A)$ is bounded in the operator norm as $\lambda \to \infty$.
	
	If there exists a complex number $\mu_0 \in \resSet(A)$ 
	and an integer $k \ge 1$ such that the operator $\Res(\mu_0,A)^k: E \to E$ is compact, 
	then $R(\lambda,A)$ is compact for each $\lambda \in \resSet(A)$.
\end{proposition}

While the proposition is certainly well-known to experts in operator theory,
we could not find it in precisely this form in the literature. 
Closely related statements, though, are shown in 
\cite[Lemma~3.4]{dePagter1989} and \cite[Proposition~4.9]{ArendtBukhvalov1994},
and the proof of Proposition~\ref{prop:compact-resolvent-by-compact-power} is very similar.
For the convenience of the reader, we include the details;
the essence of the argument is the following lemma.

\begin{lemma}
	\label{lem:compact-power-convergence}
	Under the assumptions of Proposition~\ref{prop:compact-resolvent-by-compact-power}, 
	for every $\lambda \in \resSet(A)$ we have 
	\[
		\big(\mu \Res(\mu,A)\big)^k \Res(\lambda,A) \to \Res(\lambda,A)
	\]
	with respect to the operator norm as $\bbR \ni \mu \to \infty$.
\end{lemma}
	
\begin{proof}
	Fix $\lambda \in \resSet(A)$.
	The resolvent equation yields
	\begin{align*}
		\mu \Res(\mu,A) \Res(\lambda,A) 
		= 
		\frac{\mu}{\mu - \lambda} \big( \Res(\lambda,A) - \Res(\mu,A)\big)
	\end{align*}
	for all $\mu \in \resSet(A)$.
	If $\mu$ is real and tends to $\infty$, this expression converges to $\Res(\lambda,A)$ 
	since $\mu \Res(\mu,A)$ stays bounded.
	
	Now we proceed by induction: 
	assume that we have already proved the claim for an integer $k \ge 1$.	
	For all $\mu \in \resSet(A)$ we can write the expression
	\begin{align*}
		\big(\mu \Res(\mu,A)\big)^{k+1} \Res(\lambda,A) - \Res(\lambda,A)
	\end{align*}
	as
	\begin{align*}
		\mu \Res(\mu,A) \left( \big(\mu \Res(\mu,A)\big)^k \Res(\lambda,A) - \Res(\lambda,A) \right) 
		+ 
		\Big( \mu \Res(\mu,A) \Res(\lambda,A) - \Res(\lambda,A) \Big).
	\end{align*}
	If $\mu$ is real and tends to $\infty$, the first summand converges to $0$
	by the induction hypothesis since $\mu \Res(\mu, A)$ remains bounded;
	the second summand also tends to $0$
	-- this is simply the case $k=1$ that we have already treated.
\end{proof}

\begin{proof}[Proof of Proposition~\ref{prop:compact-resolvent-by-compact-power}]
	For each $\mu \in \resSet(A)$, we obtain using the resolvent equation that
	\[
		\Res(\mu,A)^k = R(\mu_0,A)^k\left( \id + (\mu_0-\mu)\Res(\mu,A)\right)^k,
	\]
	which implies $\Res(\mu,A)^k$ is also compact. 
	So for a fixed $\lambda \in \resSet(A)$, Lemma~\ref{lem:compact-power-convergence} tells us that $\Res(\lambda,A)$ is the operator norm limit 
	of compact operators and is 
	hence compact.
\end{proof}

\section{Estimating the resolvent and its powers}
\label{section:bounds-of-resolvents}

Both the assumptions and the conclusions of our Theorem~\ref{thm:main} 
contain estimates for the resolvent $\Res(\argument,A)$ at certain points of $\resSet(A)$. 
In order to (i) be able to check that the assumptions are satisfied 
and (ii) derive consequences of the conclusions of the theorem, 
it is worthwhile to check how those estimates carry over from some points in $\resSet(A)$ 
to further points in $\resSet(A)$; this is the content of the results in this section.

First, we show that a certain two-sided estimate for the resolvent 
at a point is sufficient to obtain the same estimate 
for all real numbers in the resolvent set and for all powers of the resolvent.

\begin{proposition}
	\label{prop:extend-domination-two-sided}
	Let $A: E \supseteq \dom{A} \to E$ be a densely defined, 
	closed, and real linear operator on a complex Banach lattice $E$.
	Let $u \in E_+$ and let $\varphi \in E'$ be strictly positive.
	If there exists a real number $\mu_0 \in \resSet(A)$ such that
	\begin{align*}
		-u \otimes \varphi \preceq \Res(\mu_0,A) \preceq u \otimes \varphi,
	\end{align*}
	then we have
	\begin{align*}
		-u \otimes \varphi \preceq \Res(\mu,A)^n \preceq u \otimes \varphi
	\end{align*}
	for all real numbers $\mu \in \resSet(A)$ and all exponents $n \in \bbN$.
\end{proposition}

For the proof of the proposition, we will use the following 
well-known finite sum expansion of the resolvent. 
We state it explicitly here since we will need the same expansion several times later on 
in this and in the next section.

\begin{lemma}
	\label{lem:finite-expansion-of-resolvent}
	Let $A: E \supseteq \dom{A} \to E$ be a closed linear operator on a complex Banach space $E$
	and let $\mu_0, \mu$ be two complex numbers in the resolvent of $A$.
	Then we have
	\begin{align*}
		\Res(\mu,A)
		=
		\sum_{k=0}^{n-1}(\mu_0-\mu)^k \Res(\mu_0,A)^{k+1} + (\mu_0-\mu)^{n}\Res(\mu_0,A)^n \Res(\mu,A)
	\end{align*}
	for each integer $n \in \bbN_0$.
\end{lemma}

\begin{proof}
	This follows by applying the resolvent equation
	\begin{align*}
		\Res(\mu,A) = \Res(\mu_0,A) + (\mu_0-\mu)\Res(\mu_0,A)\Res(\mu,A)
	\end{align*}
	$n$ times.
\end{proof}

\begin{proof}[Proof of Proposition~\ref{prop:extend-domination-two-sided}]
	Fix a real number $\mu \in \resSet(A)$. 
	We apply Lemma~\ref{lem:finite-expansion-of-resolvent} for $n=2$, which gives
	\begin{align*}
		\Res(\mu,A)
		& = 
		\Res(\mu_0,A) + (\mu_0-\mu)\Res(\mu_0, A)^2 + (\mu_0-\mu)^2\Res(\mu_0, A)\Res(\mu,A)\Res(\mu_0, A).
	\end{align*}
	By assumption, the operator $\Res(\mu_0,A)$ satisfies 
	the equivalent assertions of Proposition~\ref{prop:char-of-rank-1-domination}, 
	and hence, so do $\Res(\mu_0, A)^2$ and $\Res(\mu_0, A)\Res(\mu,A)\Res(\mu_0, A)$ 
	according to Proposition~\ref{prop:rank-1-domination-composition}.
	
	So we conclude from the above formula for $\Res(\mu,A)$ that this operator also satisfies 
	the equivalent assertions of Proposition~\ref{prop:char-of-rank-1-domination}.
	Once again applying Proposition~\ref{prop:rank-1-domination-composition},
	 we conclude that all powers of $\Res(\mu,A)$ must also satisfy the claimed estimate.
\end{proof}

Let us show how Propositions~\ref{prop:extend-domination-two-sided}
and~\ref{prop:domination-imlies-estimate-for-product} can be combined to obtain 
a two-sided resolvent estimate for self-adjoint operators, 
provided that their corresponding form domain is contained in a principal ideal:

\begin{proposition}
	\label{prop:form-domain-estimates}
	Let $(\Omega,\mu)$ be a 
	$\sigma$-finite measure space and $E=L^2(\Omega,\mu)$. 
	Let $A: E \supseteq \dom{A}\to E$ be a self-adjoint and real operator 
	whose spectral bound $\spb(A)$ satisfies $\spb(A) < \infty$ and which is
	associated with a symmetric bilinear form ${a: \dom{a} \times \dom{a} \to \bbC}$,
	whose form domain we denote by $\dom{a} \subseteq E$.
	
	If $0 \le u \in E$ is a function that is strictly positive almost everywhere
	and the form domain satisfies $\dom{a} \subseteq E_u$, then 
	\[
		-u \otimes u \preceq \Res(\mu,A) \preceq u\otimes u.
	\]
	for each real number $\mu$ in the resolvent set of $A$.
\end{proposition}

\begin{proof}
	By Proposition~\ref{prop:extend-domination-two-sided},
	it suffices to prove the claim for an arbitrary real number $\mu > \spb(A)$,
	so fix such a number.
	The resolvent $\Res(\mu,A)$ is a positive-definite self-adjoint operator 
	and therefore has a unique positive definite square root $\Res(\mu,A)^{1/2}$. 
	Self-adjointness together with \cite[Theorem~8.1]{Ouhabaz2005} implies that 
	\[
		\Res(\mu,A)^{1/2} E = \dom{a} \subseteq E_u.
	\]
	Since the square root $\Res(\mu,A)^{1/2}$ is also self-adjoint,
	the operator 
	\[
		\Res(\mu,A) = \Res(\mu,A)^{1/2} \Res(\mu,A)^{1/2}
	\]
	is sandwiched between multiplies of $-u \otimes u$ and $u \otimes u$ 
	because of Proposition~\ref{prop:domination-imlies-estimate-for-product}
	(applied to $\varphi = u$).
\end{proof}

According to Proposition~\ref{prop:extend-domination-two-sided}, 
a two-sided estimate of the resolvent carries over 
from one point in $\resSet(A) \cap \bbR$ to all points in this set
and to all powers of the resolvent. 
The situation becomes more subtle if we only consider one-sided estimates. 
The estimate still carries over to the powers of the resolvent,
provided that the two conditions specified within Setting~\ref{sett:general} are satisfied;
this is the content of the following Proposition~\ref{prop:domination-powers-one-sided}. 
In the subsequent Theorem~\ref{thm:domination-interval-one-sided},
we shall see that, under the same conditions, the bounds also carries over to other values of $\mu$; 
however, we are only able to prove this for $\mu$ on one side of $\mu_0$.

\begin{proposition}
	\label{prop:domination-powers-one-sided}
	Assume that both the domination and the spectral condition are satisfied in Setting~\ref{sett:general}. 
	Let $\mu_0$ be a real number in $\resSet(A)$.
	\begin{enumerate}[\upshape (a)]
		\item 
		If $\Res(\mu_0,A) \succeq -u \otimes \varphi$, then
		$
			\Res(\mu_0,A)^n \succeq -u \otimes \varphi
		$
		for all $n \in \bbN$.
		
		\item 
		If $\Res(\mu_0,A) \preceq u \otimes \varphi$, then
		$
			(-1)^{n-1} \Res(\mu_0,A)^n \preceq u \otimes \varphi
		$
		for all $n \in \bbN$.
	\end{enumerate}
\end{proposition}
\begin{proof}
	There is no loss of generality in assuming that the eigenvalue $\lambda_0$ equals $0$.
	Moreover, it suffices to prove~(a); assertion~(b) then follows from~(a)
	by replacing $A$ with $-A$ and $\mu_0$ with $-\mu_0$.
	
	Observe that, as $v$ is an eigenvector of $A$ corresponding to $0$, 
	therefore, it is also an eigenvector of $A^{m_1}$ corresponding to the eigenvalue $0$. 
	In particular we have $v\in \dom{A^{m_1}}\subseteq E_u$ which
	together with the spectral condition implies that ${v \preceq u \preceq v}$. 
	Similarly, $\psi \preceq \varphi \preceq \psi$. 
	
	We may rescale $\psi$ and $v$ such that $\duality{\psi}{v} = 1$. 
	By Proposition~\ref{prop:spectral-properties}, 
	the eigenvalue $0$ is a pole of the resolvent $\Res(\argument,A)$ 
	and the associated spectral projection is given by $P:=v\otimes \psi$. 
	Thus, from our observation, we have $P \preceq u \otimes \varphi \preceq P$,
	so $u \otimes \varphi$ can be replaced with $P$ in all estimates 
	that we assume or that we want to show.
	In particular, the assumption in~(a) can be restated as $\Res(\mu_0,A) \succeq -P$.
	
	Now we distinguish the following two cases:	
	
	\emph{First case: $\mu_0 > 0$.}
	Then there exists a number $c \in (0,\infty)$ such that we have $\mu_0\Res(\mu_0,A) \ge -cP$.
	Hence, for $n \ge 1$,
	\begin{align*}
		0 
		\le 
		\big( \mu_0 \Res(\mu_0,A) + c P \big)^n
		& =
		\mu_0^n \Res(\mu_0, A)^n
		+
		\sum_{k=0}^{n-1} \binom{n}{k} \mu_0^k \Res(\mu_0,A)^k c^{n-k}P 
		\\
		& =
		\mu_0^n \Res(\mu_0, A)^n
		+
		\sum_{k=0}^{n-1} \binom{n}{k} c^{n-k} P;
	\end{align*}
	for the second equality we used that $\mu_0 \Res(\mu_0, A) P = P$.
	It now follows that, indeed, $\Res(\mu_0, A)^n \succeq -P$.
	
	\emph{Second case: $\mu_0 < 0$.}
	Then there exists a number $c \in (0,\infty)$ such that we have $\mu_0\Res(\mu_0,A) \le cP$.
	Moreover, we can -- and will -- choose $c$ to be strictly larger than $2$.
	Then, for $n \ge 1$,
	\begin{align*}
		0 \le 
		\big( cP - \mu_0 \Res(\mu_0, A) \big)^n
		& = 
		(-\mu_0)^n \Res(\mu_0, A)^n
		+
		\sum_{k=0}^{n-1} \binom{n}{k} (-\mu_0)^k \Res(\mu_0,A)^k c^{n-k}P.
	\end{align*}
	Again using $\mu_0 \Res(\mu_0, A) P = P$, the sum on the right simplifies to
	\begin{align*}
		\sum_{k=0}^{n-1} \binom{n}{k} (-1)^k c^{n-k} P
		& =
		(c-1)^n P - (-1)^n P
		=
		c' P,
	\end{align*}
	where the number $c' := (c-1)^n - (-1)^n$ is strictly positive since $c > 2$.
	Thus, we have shown that
	\begin{align*}
		-c' P \le (-\mu_0)^n \Res(\mu_0, A)^n,
	\end{align*}
	which proves the claim as $-\mu_0 > 0$.
\end{proof}

In the next theorem we show that 
the lower estimate $\Res(\mu_0,A) \succeq -u \otimes \varphi$
extends to all $\mu$ on the left of $\mu_0$, 
while the upper estimate $\Res(\mu_0,A) \preceq u \otimes \varphi$
extends to all $\mu$ on the right of $\mu_0$.

\begin{theorem}
	\label{thm:domination-interval-one-sided}
	Assume that both the domination and the spectral condition are satisfied in Setting~\ref{sett:general}. 
	Let $\mu_0, \mu$ be real numbers in $\resSet(A)$.
	\begin{enumerate}[\upshape (a)]
		\item 
		If $\Res(\mu_0,A) \succeq -u \otimes \varphi$ and $\mu\le \mu_0$, then
		\[
			\Res(\mu,A) \succeq -u \otimes \varphi 
		\]
		(and hence even $\Res(\mu,A)^n \succeq -u \otimes \varphi$ for all all $n\in \bbN$).
		
		\item 
		If $\Res(\mu_0,A) \preceq u \otimes \varphi$ and $\mu \ge \mu_0$, then
		\[
			\Res(\mu,A) \preceq u \otimes \varphi
		\]
		(and hence even $(-1)^{n-1} \Res(\mu,A)^n \preceq u \otimes \varphi$ for all $n \in \bbN$).
	\end{enumerate}
\end{theorem}

\begin{proof}
	We only need to prove the result for $n=1$, since the estimates for the higher powers 
	then follow immediately from Proposition~\ref{prop:domination-powers-one-sided}.
	Moreover, it is sufficient to prove~(a); assertion~(b) will follow by a sign flip again.
	
	Let $\mu$ be real number in $\resSet(A)$ such that $\mu\le \mu_0$.
	Because of the domination condition, we have both ${\Res(\mu_0,A')^{m_2} E'\subseteq (E')_{\varphi}}$
	and $\Res(\mu_0,A)^{m_1} E \subseteq E_u$.
	Hence, Proposition~\ref{prop:domination-imlies-estimate-for-product} shows that the operator
	\[
		\Res(\mu_0,A)^{m}\Res(\mu,A)
		=
		\Res(\mu_0,A)^{m_1}\Res(\mu,A)\Res(\mu_0,A)^{m_2}
	\]
	satisfies the equivalent assertions of Proposition~\ref{prop:char-of-rank-1-domination}, 
	i.e., we have
	\[
		-u \otimes \varphi \preceq \Res(\mu_0,A)^{m}\Res(\mu,A)\preceq u \otimes \varphi.
	\]
	Now we use the finite expansion of the resolvent 
	from Lemma~\ref{lem:finite-expansion-of-resolvent},
	namely
	\begin{align*}
		\Res(\mu,A)
		& =
		\sum_{k=0}^{m-1}(\mu_0-\mu)^k \Res(\mu_0,A)^{k+1} + (\mu_0-\mu)^{m}\Res(\mu_0,A)^{m}\Res(\mu,A).
	\end{align*}
	As $\mu \le \mu_0$, all operators under the sum on the left 
	dominate a negative multiple of $u \otimes \varphi$ 
	according to Proposition~\ref{prop:domination-powers-one-sided}(a);
	and the summand on the right does the same as shown in the previous paragraph.
\end{proof}

Let us sum up what the main idea in the previous proof was:
in the finite expansion formula for $\Res(\mu,A)$ from Lemma~\ref{lem:finite-expansion-of-resolvent},
the number $\mu$ occurs within a resolvent only in the expression $\Res(\mu_0,A)^{m}\Res(\mu,A)$;
but due to our domination assumption, this product 
is bounded below by a negative multiple of $u \otimes \varphi$.

\begin{remark}
	One can, for instance, use Theorem~\ref{thm:domination-interval-one-sided}(b) 
	to show that certain operators cannot satisfy the uniform anti-maximum principle 
	in the conclusion of Theorem~\ref{thm:main}(b):
	If an operator $A$ satisfies this uniform anti-maximum principle, 
	then we have $\Res(\mu,A) \preceq -u \otimes \varphi \preceq u \otimes \varphi$ 
	for some $\mu$ in a left neighbourhood of $\lambda_0$; 
	hence, it follows from Theorem~\ref{thm:domination-interval-one-sided}(b) 
	that actually $\Res(\mu,A) \preceq u \otimes \varphi$ for all sufficiently large $\mu \in \bbR$. 
	Thus, operators that do not satisfy the latter inequality for all large $\mu$ 
	cannot satisfy a uniform anti-maximum principle.
	
	In Theorem~\ref{thm:anti-max-characterization} we will use such an argument
	to derive a characterization of the anti-maximum principle.
\end{remark}

The situation becomes a bit simpler than in Theorem~\ref{thm:domination-interval-one-sided} 
if the domination condition in Setting~\ref{sett:general} is satisfied for the exponents ${m_1 = m_2 = 1}$:

\begin{proposition}
	\label{prop:extend-domination-exponent-one}
	Assume that the domination condition in Setting~\ref{sett:general} is satisfied for ${m_1 = m_2 = 1}$. 
	Let $\mu_0$ be a real number in $\resSet(A)$.
	\begin{enumerate}[\upshape (a)]
		\item 
		We have $\Res(\mu_0,A) \succeq -u \otimes \varphi$ if and only if 
		$\Res(\mu,A) \succeq -u \otimes \varphi$ for all real numbers $\mu$ in the resolvent set of $A$.
		
		\item 
		We have $\Res(\mu_0,A) \preceq u \otimes \varphi$ if and only if 
		$\Res(\mu,A) \preceq u \otimes \varphi$ for all real numbers $\mu$ in the resolvent set of $A$.
	\end{enumerate}
\end{proposition}

\begin{proof}
	Let $\mu\in \resSet(A)\cap \mathbb{R}$. 
	Corollary~\ref{cor:two-sided-estimate-of-resolent-power-by-domination-assumption}
	gives us $-u \otimes \varphi \preceq \Res(\mu,A)\Res(\mu_0,A)\preceq u \otimes \varphi$.  
	Therefore, the resolvent equation
	\[
		\Res(\mu,A)=\Res(\mu_0,A) + (\mu_0-\mu)\Res(\mu,A)\Res(\mu_0,A).
	\]
	immediately implies the proposition.
\end{proof}

Thus, if we want to derive a one-sided estimate for $\Res(\mu,A)$ 
from a one-sided estimate for $\Res(\mu_0,A)$,
Proposition~\ref{prop:extend-domination-exponent-one} tells us that, if $m_1 = m_2 = 1$, then
we do not need to care whether $\mu \le \mu_0$ or $\mu \ge \mu_0$.
For the more general $m_1$ and $m_2$, we do not know whether
the conditions $\mu \le \mu_0$ and $\mu \ge \mu_0$
in Theorem~\ref{thm:domination-interval-one-sided}(a) and~(b), respectively, can be dropped.

\section{Eventual positivity of resolvents and their powers}
\label{section:eventual-positivity}

In this section, we finally come back to sufficient criteria
-- and characterizations -- of uniform (anti-)maximum principles. 
Our goals in this section are to finally prove Theorem~\ref{thm:main}, 
to give a similar result on powers of the resolvent (Theorem~\ref{thm:ev-pos-powers}) albeit with less restrictive assumptions,
and to prove a characterization of the uniform anti-maximum principle
under appropriate additional assumptions
(Theorem~\ref{thm:anti-max-characterization} and Corollary~\ref{cor:anti-max-characterization}).

The following resolvent estimate, which is a consequence of our spectral and domination assumption,
will be very useful.

\begin{lemma}
	\label{lem:resolvent-convergence}
	Let both the domination and the spectral assumption be satisfied in Setting~\ref{sett:general}
	and assume that the number $m = m_1 + m_2$ is non-zero. 
	Then there exists an open neighbourhood $U \subseteq \bbR$ of $\lambda_0$ 
	that does not contain any spectral value of $A$ except for $\lambda_0$, 
	and a function ${c: U\setminus\{\lambda_0\} \to (0,\infty)}$ with the following properties:
	
	We have $c(\mu) \to 0$ as $\mu \to \lambda_0$, and
	\begin{align*}
		-c(\mu) \, u \otimes \varphi \le (\mu-\lambda_0)^m \Res(\mu,A)^m - P \le c(\mu) \, u \otimes \varphi
	\end{align*}
	for all $\mu \in U\setminus \{0\}$; 
	here, $P := \frac{\psi \otimes v}{\duality{\psi}{v}}$ denotes the spectral projection 
	of $A$ associated to $\lambda_0$.
\end{lemma}

\begin{proof}
	Replacing $A$ with $A-\lambda_0$, we assume $\lambda_0=0$. 
	By Proposition~\ref{prop:spectral-properties}, 
	the spectral value $0$ is a first order pole of the resolvent $\Res(\argument,A)$ 
	and the corresponding spectral projection is given by $P=\frac{\psi \otimes v}{\duality{\psi}{v}}$. 
	The former implies that $\mu\Res(\mu,A)\to P$ as $\mu \to 0$ 
	in the operator norm (see \cite[Proposition~4.3.15]{ArendtBattyHieberNeubrander2011}). 
	
	Fix a real number $\lambda$ in the resolvent set of $A$. 
	The domination condition implies that the dual operator of $\Res(\lambda,A)^{m_2}$ maps into $(E')_\varphi$,
	so -- as mentioned before Proposition~\ref{prop:domination-imlies-estimate-for-product} --
	$\Res(\lambda,A)^{m_2}$ extends to
	a bounded linear operator from $E^{\varphi}$ to $E$.
	Again by the domination condition and due to the closed graph theorem,
	we also have $\Res(\lambda,A)^{m_1} \in \calL(E,E_u)$.
	
	For all real numbers $\mu \in \resSet(A)$ we get from 
	$\lambda\Res(\lambda,A)P=P$, together with the resolvent equation that
	\[
		\left(\mu \Res(\mu,A)\right)^{m}-P 
		=
		\Res(\lambda,A)^{m_1} \, \big( S(\mu)^{m} - \lambda^m P \big) \, \Res(\lambda,A)^{m_2};
	\]
	where $S(\mu):=\mu+(\lambda-\mu)\mu\Res(\mu,A)$ is a bounded operator on $E$.
	The middle term on the right $S(\mu)^{m} - \lambda^m P$ converges to $0$ 
	with respect to the operator norm on $\calL(E)$ as $\mu \to 0$
	(here we have used that $m \ge 1$) and so we obtain 
	\[
		c(\mu) 
		:= 
		\norm{\left(\mu \Res(\mu,A)\right)^{m}-P}_{E^{\varphi}\to E_u}
		\to 0
	\]
	as $\mu \to 0$.
	Finally, the numbers $c(\mu)$ satisfy the desired estimate
	according to Remark~\ref{rem:char-of-rank-1-domination-constants}.
\end{proof}

As a direct application of Lemma~\ref{lem:resolvent-convergence}, 
we can prove a first eventual positivity result. 
It deals with the powers of the resolvent rather than the resolvent itself. Recall that we have set $m:=m_1+m_2$.

\begin{theorem}
	\label{thm:ev-pos-powers}
	Let both the domination and the spectral assumption be satisfied in Setting~\ref{sett:general}. 
	Then all spectral values of $A$ are isolated, 
	and the following assertions hold:
	\begin{enumerate}[\upshape (a)]
		\item 
		For all $\mu$ in a right neighbourhood of $\lambda_0$ we have 
		\[
			u \otimes \varphi \preceq \Res(\mu,A)^{m} \preceq u \otimes \varphi.
		\]
		
		\item 
		For all $\mu$ in a left neighbourhood of $\lambda_0$ we have
		\[
			-u \otimes \varphi \preceq (-1)^{m-1} \Res(\mu,A)^{m} \preceq -u \otimes \varphi.
		\]
	\end{enumerate}
\end{theorem} 
\begin{proof}
	According to Proposition~\ref{prop:spectral-properties}(a), 
	all spectral values of $A$ are isolated.

	Next we note that, if $m = 0$, then $E = E_u$ and $(E')_\varphi = E'$,
	so by Proposition~\ref{prop:domination-imlies-estimate-for-product} every
	real bounded linear operator $S$ on $E$ satisfies
	$u \otimes \varphi \preceq S \preceq u \otimes \varphi$.	
	Thus, we may assume from now on that $m \ge 1$,
	so that Lemma~\ref{lem:resolvent-convergence} can be applied.
	Let the open set $U$, the function $c$, and the spectral projection $P$ be as in the lemma. 
	Just like in the proof of Proposition~\ref{prop:domination-powers-one-sided}, 
	we observe that we have the estimate $u\otimes \varphi \preceq P \preceq u\otimes \varphi$.
	The assertions can now be deduced from
	\[
		P-c(\mu) \, u \otimes \varphi 
		\le 
		\left((\mu-\lambda_0) \Res(\mu,A)\right)^{m} 
		\le 
		P+c(\mu) \, u \otimes \varphi
	\]
	and $c(\mu)\to 0$ as $\mu \to \lambda_0$.
\end{proof}

Our main result, Theorem~\ref{thm:main}, essentially says that 
we can drop the exponent $m$ in the conclusion of Theorem~\ref{thm:ev-pos-powers} 
if a certain a priori estimate holds for the resolvent. 
Let us finally prove this result; 
the argument is quite similar to the proof of Lemma~\ref{lem:resolvent-convergence}
and Theorem~\ref{thm:ev-pos-powers}:

\begin{proof}[Proof of Theorem~\ref{thm:main}]
	It is a consequence of the domination assumption that every spectral value 
	of $A$ is isolated, see Proposition~\ref{prop:spectral-properties}(a).
	
	By a change of signs, it suffices to prove part~(a) of the theorem. 
	To this end, without loss of generality, assume that $\lambda_0=0$. Let  
	\begin{align*}
		R_{\mu} 
		& :=
		(\mu_0-\mu)^m \Res(\mu_0,A)^m\mu \Res(\mu,A)
		\\
		& =
		(\mu_0-\mu)^m \Res(\mu_0,A)^{m_1} \mu \Res(\mu,A) \Res(\mu_0,A)^{m_2}
	\end{align*}
	for all real numbers $\mu$ in the resolvent set of $A$ 
	and let $P$ be the spectral projection of $A$ corresponding to $0$. 
	As in the proof of Lemma~\ref{lem:resolvent-convergence}, we can see that
	\begin{align*}
		c(\mu) := \norm{R_{\mu} - P}_{E^\varphi \to E_u} \to 0
	\end{align*}
	as $\mu \to 0$ (but this time the convergence is also true for $m=0$).
	Again, we will use that
	\[
		R_{\mu}\geq P-c(\mu) (u\otimes \varphi)
	\]
	for every $\mu\in \resSet(A) \cap \bbR$; 
	this is an immediate consequence of Remark~\ref{rem:char-of-rank-1-domination-constants}.	 
	Since $P \succeq u \otimes \varphi$, there exists $c>0$ such that 
	$R_{\mu} \geq (c-c(\mu)) u\otimes \varphi$ for all $\mu \in \resSet(A) \cap \bbR$. 
	
	Now, let us consider numbers $\mu \in (0, \mu_0] \cap \resSet(A)$. 
	Then the finite expansion of the resolvent from Lemma~\ref{lem:finite-expansion-of-resolvent}, 
	yields
	\begin{align*}
		\mu\Res(\mu,A)
		& = 
		\mu\sum_{k=0}^{m-1} (\mu_0-\mu)^k\Res(\mu_0,A)^{k+1}+R_{\mu}
		\\
		& \geq 
		\mu \sum_{k=0}^{m-1} (\mu_0-\mu)^k \Res(\mu_0,A)^{k+1} + (c-c(\mu))u\otimes \varphi.
	\end{align*}
	The assumption $\Res(\mu_0,A)\succeq -u\otimes \varphi$ 
	and Proposition~\ref{prop:domination-powers-one-sided}(a)
	ensure that we have $\Res(\mu_0,A)^{k+1} \ge - c_k u\otimes \varphi$ 
	for all $k \in \{0, \dots, m-1\}$ and positive constants $c_0, \dots, c_{m-1}$. 
	Since the coefficients $\mu (\mu_0-\mu)^k$ are all positive, we deduce that
	\[
		\mu \Res(\mu,A) \geq - d(\mu) \, u\otimes \varphi + (c-c(\mu)) \, u\otimes \varphi;
	\]
	here, the number $d(\mu)$ is defined as $d(\mu) := \mu \sum_{k=0}^{m-1} (\mu_0-\mu)^k c_k$ 
	and thus converges to $0$ as $\mu\to 0$. 
	It follows that $\Res(\mu,A)\succeq u\otimes\varphi$ 
	for all $\mu$ in a sufficiently small right neighbourhood of $0$.
\end{proof}

We end this section with a very useful consequence of Theorems~\ref{thm:main}
and~\ref{thm:domination-interval-one-sided}:
a characterization for the uniform anti-maximum principle under the condition
that a certain resolvent estimate holds at a single point on the right of $\lambda_0$.
As before, all spectral values of $A$ in the following theorem and in its corollary 
are isolated as a consequence of Proposition~\ref{prop:spectral-properties}(a).

\begin{theorem}
	\label{thm:anti-max-characterization}
	Let both the domination and the spectral assumption be satisfied in Setting~\ref{sett:general}.
	If there exists a number $\mu_1>\lambda_0$ in the resolvent set of $A$ 
	such that $\Res(\mu_1,A)\succeq -u\otimes \varphi$, 
	then the following are equivalent:
	\begin{enumerate}[ref=(\roman*)]
		\item\label{item:anti-max-characterization:anti-max-strong} 
		The \emph{uniform anti-maximum principle} holds, i.e., we have
		\[
			\Res(\mu,A) \preceq - u \otimes \varphi
		\]
		for all $\mu$ in a left neighbourhood of $\lambda_0$.
				
		\item \label{item:anti-max-characterization:negative-one-point}
		There exists a number $\mu_0 < \lambda_0$ in the resolvent set of $A$ 
		at which we have $\Res(\mu_0,A) \leq 0$.
		
		\item\label{item:anti-max-characterization:one-point} 
		We have $\Res(\mu_1,A) \preceq u\otimes \varphi$.
	\end{enumerate}
\end{theorem}

\begin{proof}
	\Implies{item:anti-max-characterization:anti-max-strong}{item:anti-max-characterization:negative-one-point}
	This implication is of course, obvious.
	
	\Implies{item:anti-max-characterization:negative-one-point}{item:anti-max-characterization:one-point}
	We have $\Res(\mu_0,A) \leq 0 \preceq u\otimes \varphi$, which, along with Theorem~\ref{thm:domination-interval-one-sided}(b) gives
	the desired estimate $\Res(\mu_1,A) \preceq u\otimes \varphi$.
	
	\Implies{item:anti-max-characterization:one-point}{item:anti-max-characterization:anti-max-strong} 
	Because of the two-sided estimate
	\[
		- u\otimes \varphi \preceq \Res(\mu_1,A) \preceq u\otimes \varphi,
	\]
	we can employ Proposition~\ref{prop:extend-domination-two-sided} to obtain 
	$\Res(\mu_0,A)\preceq u\otimes \varphi$ for all $\mu_0 \in \resSet(A) \cap \bbR$
	and thus, in particular, for all real number $\mu_0 \in \resSet(A)$ 
	that are located on the left of $\lambda_0$.
	Thus $A$ satisfies the uniform anti-maximum principle 
	by Theorem~\ref{thm:main}\ref{thm:main:item:anti-max}.
\end{proof}

Theorem~\ref{thm:anti-max-characterization} 
can be seen as an abstract and very general version of \cite[Theorem~1]{GrunauSweers2001}, 
where the uniform anti-maximum principle for polyharmonic operators was characterized. 
In particular, \cite[Theorem~1]{GrunauSweers2001} can be derived 
with the help of Theorem~\ref{thm:anti-max-characterization} 
(see Theorem~\ref{thm:polyharmonic-ball-anti-max} below for details).

A special case of the previous theorem is the situation
where the resolvent at $\mu_1$ is even a positive operator.
It is worthwhile to state this special case explicitly in an extra corollary
-- due to the prevalence of the assumption $\Res(\mu_1,A)\geq 0$ in applications
and since it is particularly relevant for the important case of generators of positive semigroups.

\begin{corollary}
	\label{cor:anti-max-characterization}
	Let both the domination and the spectral assumption be satisfied in Setting~\ref{sett:general}. 
	If there exists $\mu_1>\lambda_0$ in the resolvent set of $A$ such that $\Res(\mu_1,A)\geq 0$, 
	then the following are equivalent.
	\begin{enumerate}[ref=(\roman*)]
		\item 
		The \emph{uniform anti-maximum principle} holds, i.e., we have
		\[
			\Res(\mu,A) \preceq - u \otimes \varphi
		\]
		for all $\mu$ in a left neighbourhood of $\lambda_0$.
		
		\item
		There exists a number $\mu_0 < \lambda_0$ in the resolvent set of $A$ 
		at which we have $\Res(\mu_0,A) \leq 0$.
				
		\item 
		We have $\Res(\mu_1,A) \preceq u\otimes \varphi$.
	\end{enumerate}
\end{corollary}

Let us stress that, if $A$ generates a positive $C_0$-semigroup, 
then the resolvent is positive on the right of the spectral bound; 
thus, Corollary~\ref{cor:anti-max-characterization} gives a characterization 
of the anti-maximum principle for such operators (for $\lambda_0=\spb(A)$). 

By a change of signs, one can formulate an analogue of Theorem~\ref{thm:anti-max-characterization} 
and Corollary~\ref{cor:anti-max-characterization} for the maximum principle as well.

\section{Examples and applications}
\label{section:applications}

In this section, we consider a variety of concrete differential operators
and demonstrate how our results can be used to prove or disprove
that they satisfy a uniform (anti-)maximum principle.
Our choice of examples does by no means cover all operators for which (anti-)maximum principles
have been proved in the literature.
It is rather our goal to demonstrate the applicability and the strength of our abstract approach
by discussing merely a selection of examples, namely of the following three types:

\begin{enumerate}[(1)]
	\item
	We revisit a few easy and well-understood toy examples to demonstrate how our machinery works in principle
	(Proposition~\ref{prop:laplace-operator-on-interval}).
	
	\item 
	We prove (anti-)maximum principles for several (classes of) operators
	for which these results seem to be new
	(Propositions~\ref{prop:laplace-in-intervall-non-local-symmetric} 
	and~\ref{prop:laplace-in-intervall-non-local-thermostat}, as well as
	Subsections~\ref{subsection:application-metric-graph}, 
	\ref{subsection:applications:odd-order}, and~\ref{subsection:application-delay-operator})
	
	\item 
	We look at several operators for which (anti-)maximum principles are already known,
	and we illustrate how our abstract results can be used to facilitate the proofs of these results
	(Subsection~\ref{subsection:applications:polyharmonic-ball}).
\end{enumerate}

\subsection{Laplace operators on intervals}

In this subsection, we prove anti-maximum principles for the Laplace operator
on an interval with various boundary conditions.

As a warm-up, we first revisit a few simple and classical boundary conditions
in Proposition~\ref{prop:laplace-operator-on-interval}.
There is nothing novel in this proposition, 
but it serves as a nice example to illustrate how our results can be applied.
On a more serious note, we then proceed with anti-maximum principles
for certain non-local boundary conditions
(Propositions~\ref{prop:laplace-in-intervall-non-local-symmetric}
and~\ref{prop:laplace-in-intervall-non-local-thermostat}).

Throughout, let $(\alpha,\beta) \subseteq \bbR$ be a non-empty and bounded open interval.
On $E := L^2(\alpha,\beta)$ we consider the Laplace operator
\begin{equation}
	\label{eq:laplace-operator-on-interval-bc}
	\begin{split}
		\dom{\Delta} &:=\{f\in H^2(\alpha,\beta): \; \BoundCond{f}\},\\
			\Delta f &= f'',
	\end{split}
\end{equation}
where $\BoundCond{f}$ denote the boundary conditions 
imposed on the vectors $f$ in the domain of $\Delta$.

Firstly, we go over three classical examples: 
the Laplace operator on the interval $(0,1)$ 
subject to Dirichlet, Neumann, and periodic boundary conditions.
The following proposition contains both positive and negative results about
the anti-maximum principle for these operators.

\begin{proposition} 
	\label{prop:laplace-operator-on-interval}
	Let $(\alpha,\beta) = (0,1)$.
	Then the following assertions hold for the Laplace operator $\Delta$
	defined in~\eqref{eq:laplace-operator-on-interval-bc}
	with Dirichlet, Neumann, and periodic boundary conditions, respectively.
	\begin{enumerate}[\upshape (a)]
		\item 
		\emph{No uniform anti-maximum principle for Dirichlet boundary conditions:} 
		Let $\BoundCond{f} : \Leftrightarrow f(0) = f(1) = 0$. \\
		Then there is no real number $\mu \in \resSet(\Delta)$ for which we have
		$\Res(\mu,\Delta) \le 0$
		(in particular, $\Delta$ does not satisfy a uniform anti-maximum principle 
		at its spectral bound $\spb(\Delta) = -\pi^2$).
		
		\item 
		\emph{Uniform anti-maximum principle for Neumann boundary conditions:} 
		Let $\BoundCond{f} : \Leftrightarrow f'(0) = f'(1) = 0$. \\
		Then for every $\mu$ in a left neighbourhood of $\spb(\Delta) = 0$, 
		we have that $\Res(\mu,\Delta) \preceq - \one \otimes \one$.
		
		\item 
		\emph{Uniform anti-maximum principle for periodic boundary conditions:} 
		Let $\BoundCond{f} : \Leftrightarrow \big(f(0) = f(1) \text{ and } f'(0) = f'(1)\big)$. \\
		Then for every $\mu$ in a left neighbourhood of $\spb(\Delta) = 0$, 
		we have that $\Res(\mu,\Delta) \preceq - \one \otimes \one$.
	\end{enumerate}  
\end{proposition}

\begin{proof}
	(b) and (c): 	
	Let $\one$ denote the constant one function in $E$.
	Since $\Delta$ is self-adjoint and $\dom{\Delta}$ is a subset of $L^\infty(0,1) = (L^2(0,1))_{\one}$,
	the domination assumption in Setting~\ref{sett:general}
	is satisfied for $u = \varphi = \one$ and for $m_1 = m_2 = 1$.
	
	Due to the boundary conditions, $\one$ is a geometrically simple eigenvalue 
	for the spectral bound $\spb(\Delta) = 0$.
	Whence, the spectral assumption in Setting~\ref{sett:general} is also satisfied.
	
	Since the operator $\Delta$ generates a positive $C_0$-semigroup,
	its resolvent is positive at every point in $(0,\infty)$.
	Thus, the characterization of the anti-maximum principle from
	Corollary~\ref{cor:anti-max-characterization} is applicable.
	In particular, we only need to show that the estimate 
	$\Res(1,\Delta) \preceq \one \otimes \one$ holds.
	
	But this estimate follows from Proposition~\ref{prop:form-domain-estimates} 
	since the self-adjoint operator $\Delta$ is associated with a symmetric bilinear form $a$
	whose form domain $\dom{a}$ is a subset of $H^1(0,1)$ and
	thus a subset of $L^\infty(0,1) = (L^2(0,1))_{\one}$.
	
	(a) 
	For Dirichlet boundary conditions, the spectral and domination assumptions
	are again satisfied in Setting~\ref{sett:general} for $m_1 = m_2 = 1$,
	but this time for the function $u = \varphi$ given by $u(x) := \sin(\pi x)$, 
	which is the first eigenfunction of $\Delta$.
	
	Since $\Delta$ generates a positive $C_0$-semigroup \cite[Example~11.14(b)]{BatkaiKramarRhandi2017},
	we can again apply the characterization of the anti-maximum principle from
	Corollary~\ref{cor:anti-max-characterization}.
	
	Now assume towards a contradiction 
	that there exists a number $\mu \in \resSet(\Delta) \cap \bbR$
	such that $\Res(\mu,\Delta) \le 0$.
	Since $\Delta$ generates a positive semigroup, we must then have $\mu < \spb(\Delta) = -\pi^2$
	and hence, it follows from Theorem~\ref{thm:domination-interval-one-sided}(b)
	that $\Res(0,\Delta) \preceq u \otimes u$.
	However, it is not difficult to compute $\Res(0,\Delta)$ explicitly:
	for each $f \in L^2(0,1)$ the formula
	\[
		\Res(0,\Delta)f(x)= \int_0^x y(1-x) f(y) \dx y + \int_x^1 x(1-y) f(y) \dx y
 	\]
 	holds for every $x\in (0,1)$.
 	We can now infer that the integral kernel (or the \emph{Green's function}) $G: (0,1) \times (0,1) \to \bbR$ of $\Res(0,\Delta)$ is given by
 	\begin{align*}
 		G(x,y) =
 		\begin{cases}
 			y(1-x) \quad & \text{if } y \le x, \\
 			x(1-y) \quad & \text{if } y \ge x.
 		\end{cases}
 	\end{align*}
 	Hence, $G$ is not dominated by a positive multiple of $(x,y) \mapsto \sin(\pi x) \sin(\pi y)$,
 	as can be seen by considering the behaviour of $G$ on the diagonal of the square $(0,1)^2$
 	as $(x,x)$ approaches the corner $(0,0)$.
	As a result, $\Res(0,\Delta) \not\preceq u \otimes u$, which gives us the desired contradiction. 
\end{proof}

One can also use our results to show maximum principles for the 
boundary conditions in the previous theorem.
However, for those boundary conditions, lower estimates of the form
$\Res(\mu, \Delta) \succeq u \otimes u$ 
(where $u = \sin(\pi \argument)$ for Dirichlet boundary conditions
and $u = \one$ for Neumann or periodic boundary conditions)
are well-known to hold even for all $\mu > \spb(\Delta)$.
More generally, for many generators of positive $C_0$-semigroups,
such estimates for all $\mu$ that are larger than the spectral bound
can be shown by means of ultracontractivity properties of the semigroup;
see for instance \cite[Theorem~4.2.5]{Davies1989}.

We again emphasize that the results in Proposition~\ref{prop:laplace-operator-on-interval}
are hardly surprising and are merely a first and simple illustration of our methods.
Part~(b) of the proposition can, for instance, be found 
-- for more general coefficients and Robin boundary conditions, actually -- 
in \cite[Theorem~5.3]{ClementPeletier1979}. 
In the introduction of the same paper,
it was also mentioned that part~(a) of Proposition~\ref{prop:laplace-operator-on-interval} above
can be seen by explicitly computing the resolvent of the Dirichlet Laplacian
at numbers $\mu$ that are smaller than the spectral bound.
We do not know an explicit reference where part~(c) of the proposition is proved,
but due to the spectral and domination properties of the Laplace operator with periodic boundary conditions
-- which are very similar to the case of Neumann boundary conditions --
the result is not much of a surprise.

Before we proceed to discuss more involved boundary conditions,
let us briefly comment on the anti-maximum principle of the Laplace operator in higher dimensions.
It was shown in \cite[Theorem~3.2]{ClementPeletier1979} that if $\Omega\subseteq \bbR^n$ is a bounded smooth domain, 
then a non-uniform anti-maximum principle holds at the first eigenvalue 
for the Laplacian with mixed Dirichlet and Neumann boundary conditions 
whenever the initial value lies in $L^p(\Omega)$ for $p>n$.
In the introduction of the same paper, 
it is explained why the result cannot be uniform in dimension $\ge 2$.
We also note that the non-uniform anti-maximum principle no longer holds if $p\leq n$; 
see \cite{Sweers1997}. 

Let us now get back to the Laplace operator in dimension one,
given by~\eqref{eq:laplace-operator-on-interval-bc} 
-- but this time we consider non-local boundary conditions.
More precisely, let $B \in \bbR^{2 \times 2}$; 
in the domain $\dom{\Delta}$ in~\eqref{eq:laplace-operator-on-interval-bc}
we now consider the boundary conditions
\begin{align}
	\label{eq:interval-bc-non-local}
	\BoundCond{f} 
	\quad :\Leftrightarrow \quad 
	\frac{\partial}{\partial \nu} f 
	= 
	-B 
	\begin{pmatrix}
		f(\alpha) \\
		f(\beta)
	\end{pmatrix},
\end{align}
where
\begin{align*}
	\frac{\partial}{\partial \nu} f
	:=
	\begin{pmatrix}
		-f'(\alpha) \\
		 f'(\beta)
	\end{pmatrix}
\end{align*}
denotes the outer normal derivative of $f$ at the boundary of the interval $(\alpha,\beta)$.

In the following two propositions, we deal with two special cases 
of the boundary conditions~\eqref{eq:interval-bc-non-local} 
that have already been treated in \cite[pp.\,2625--2627]{DanersGlueckKennedy2016b}.
The results developed in the current paper enable us to show 
uniform anti-maximum principles for these boundary conditions,
which are not contained in \cite{DanersGlueckKennedy2016b}. In addition, uniform eventual positivity of the corresponding semigroups was the content of \cite[Theorems~4.2 and 4.3]{DanersGlueck2018b}.

\begin{proposition}
	\label{prop:laplace-in-intervall-non-local-symmetric}
	Let $(\alpha, \beta) = (0,1)$.
	Endow the Laplace operator $\Delta$ defined in~\eqref{eq:laplace-operator-on-interval-bc}
	with the boundary conditions~\eqref{eq:interval-bc-non-local},
	where we choose the matrix $B$ as
	\[
		B =
		\begin{bmatrix}
			1 & 1 \\
			1 & 1
		\end{bmatrix}.
	\] 
	Then we have $\spb(\Delta) < 0$.
	Moreover, for all $\mu$ in a right neighbourhood of $\spb(\Delta)$,
	we have that $\Res(\mu,\Delta) \succeq \one \otimes \one$,
	and for all $\mu$ in a left neighbourhood of $\spb(\Delta)$,
	the estimate $\Res(\mu, \Delta) \preceq - \one \otimes \one$ holds.
\end{proposition}

\begin{proof}
	The spectral bound $\spb(\Delta)$ is strictly negative 
	according to \cite[Lemma~6.9]{DanersGlueckKennedy2016b}
	and it was shown in \cite[Theorem~4.2]{DanersGlueck2018b} that the domination assumption 
	in Setting~\ref{sett:general} is satisfied with $m_1=m_2=1$ and $u=\varphi=\one$. 
	Moreover, as a consequence of \cite[Theorem~6.11 and Proposition~3.1]{DanersGlueckKennedy2016b}, 
	the spectral assumption in Setting~\ref{sett:general} is also fulfilled. 
	
	The operator $\Delta$ is self-adjoint \cite[Lemma~6.9]{DanersGlueckKennedy2016b}, 
	and it is associated with a form whose form domain is $H^1(0,1)$
	(see \cite[pp.\,2625--2626]{DanersGlueckKennedy2016b}).
	Hence, Proposition~\ref{prop:form-domain-estimates} is applicable and yields 
	that $-\one \otimes \one \preceq \Res(\mu,A) \preceq \one \otimes \one$ for all real numbers $\mu>0$. 
	Proposition~\ref{prop:extend-domination-exponent-one} now implies 
	that actually the estimate holds for all real numbers $\mu$ in the resolvent set of $A$. 
	The result thus follows by Theorem~\ref{thm:main}.
\end{proof}

We note that the Laplace operator with the boundary conditions in
Proposition~\ref{prop:laplace-in-intervall-non-local-symmetric} 
is also discussed in some details in \cite[Section~3]{Akhlil2018}.
The uniform maximum principle proved above also follows 
from arguments in \cite{DanersGlueckKennedy2016b};
more precisely, the resolvent $\Res(0, \Delta)$ 
was explicitly computed in the proof of \cite[Theorem~6.11]{DanersGlueckKennedy2016b},
and from the formula for $\Res(0, \Delta)$ obtained there one can see that 
$\Res(0, \Delta) \succeq \one \otimes \one$.
By the series expansion of the resolvent about the point $0$,
one then obtains that $\Res(\mu,\Delta) \succeq \one \otimes \one$
for all $\mu \in (\spb(\Delta), 0]$.
However, the anti-maximum principle shown 
in Proposition~\ref{prop:laplace-in-intervall-non-local-symmetric} was,
to the best of our knowledge, not known before.

For our last example, we consider a situation where $\Delta$ is not self-adjoint. 
Therefore, unlike in the previous results, 
we do not have the luxury to employ Proposition~\ref{prop:form-domain-estimates}. 

\begin{proposition}
	\label{prop:laplace-in-intervall-non-local-thermostat}
	Let $(\alpha, \beta) = (0,\pi)$.
	Endow the Laplace operator $\Delta$ defined in~\eqref{eq:laplace-operator-on-interval-bc}
	with the boundary conditions~\eqref{eq:interval-bc-non-local},
	where we choose the matrix $B$ as
	\[
		B =
		\begin{bmatrix}
			0 & \beta \\
			0 & 0
		\end{bmatrix}
	\] 
	where $\beta \in (0, 1/\pi)$.
	Then $\spb(\Delta) < 0$,
	for all $\mu \in (\spb(\Delta),0]$ we have $\Res(\mu,\Delta) \succeq \one \otimes \one$,
	and for all $\mu$ in a left neighbourhood of $\spb(\Delta)$ 
	we have $\Res(\mu, \Delta) \preceq - \one \otimes \one$.
\end{proposition}

The boundary conditions in this proposition were used in \cite{GuidottiMerino1997}
in a linearisation of a thermostat model.
In \cite[Theorem~5.7]{GuidottiMerino2000}, a maximum principle was shown 
for the operator $\Delta$ with the above boundary conditions.
Our proposition shows that a uniform anti-maximum principle holds as well.

We note that before Theorem~6.10 in \cite{DanersGlueckKennedy2016b} it was claimed
that $\spb(\Delta) > 0$; 
this claim is apparently based on a confusion of $\Delta$ and $-\Delta$
(but this does not affect the arguments in the proof of \cite[Theorem~6.10]{DanersGlueckKennedy2016b}).

\begin{proof}[Proof of Proposition~\ref{prop:laplace-in-intervall-non-local-thermostat}]
	The domination assumption in Setting~\ref{sett:general}
	is satisfied for $u = \varphi = \one$ and $m_1 = m_2 = 1$
	(note that taking the adjoint of $\Delta$ simply corresponds
	to swapping the boundary points of the interval $(0,\pi)$).
	
	Since $\beta < \frac{1}{\pi}$ we, in particular, have
	that $\beta < \frac{1}{2}$.
	Using the latter inequality, one can show by a direct computation 
	that the spectral bound $\spb(\Delta)$ is strictly negative,
	and that the spectral assumption from Setting~\ref{sett:general} is also satisfied;
	see for instance \cite[Theorem~11.7.4]{Glueck2016} 
	and \cite[Section~3]{GuidottiMerino1997} for details
	(we note that the operator $A$ that occurs in parts~(b) and~(c) 
	of \cite[Theorem~11.7.4]{Glueck2016} is the same as the operator $\Delta_B$ there;
	also, note that the boundary points of the interval are swapped there
	and that the parameter $\beta$ there is minus our parameter $\beta$).
	 
	An explicit computation shows that the resolvent of $\Delta$ at the point $0$ is given by
	\[
		R(0,\Delta) f(x) 
		= 
		\int_x^{\pi} \left(\frac{1}{\beta} + x - y\right) f(y) \dx y 
		+ \frac{1}{\beta} \int_0^x f(y) \dx y
	\]
	for all $f\in L^2(0,\pi)$ and $x\in (0,\pi)$.
	We conclude that $\one \otimes \one \preceq \Res(0,\Delta) \preceq \one \otimes \one$,
	where the first estimate holds since $\beta < 1/\pi$.
	
	The series expansion of $\Res(\mu,\Delta)$ about the point $0$
	now implies the estimate $\one \otimes \one \preceq \Res(\mu, \Delta)$
	for all $\mu \in (\spb(\Delta), 0]$.
	
	In addition, Proposition~\ref{prop:extend-domination-two-sided} shows that
	$-\one \otimes \one \preceq \Res(\mu, \Delta) \preceq \one \otimes \one$ for all real numbers $\mu \in \resSet(\Delta)$
	that are smaller than $\spb(\Delta)$.
	Part~\ref{thm:main:item:anti-max} of our main result, Theorem~\ref{thm:main},
	thus shows that even $\Res(\mu, \Delta) \preceq - \one \otimes \one$ 
	for all $\mu$ in a left neighbourhood of $\spb(\Delta)$.
\end{proof}

\subsection{An anti-maximum principle for a Laplace operator on a metric graph}
\label{subsection:application-metric-graph}

In this subsection, we briefly discuss an extension of the anti-maximum principle
for the Neumann-Laplacian and the Laplace operator 
with periodic boundary conditions on an interval from 
Proposition~\ref{prop:laplace-operator-on-interval}(b) and~(c):
we now consider a Laplace operator on a finite metric graph, subject to Kirchhoff boundary conditions.
For a general reference to elliptic operators on metric graphs 
-- in particular in the context of evolution equations --
we refer to the monograph \cite{Mugnolo2014}.

Let $G = (V, E)$ be a finite, connected, and (a priori) undirected graph.
Assign a length $\ell_e$ to each edge $e \in E$.
We identify each edge $e \in E$ with a copy of the interval $[0, \ell_e]$;
in order to do so, we actually need to endow $e$ with a direction --
which is, however, arbitrary and not of relevance for what follows.

By $H^1(G)$, we denote the set of all tuples of functions
\begin{align*}
	f = (f_e)_{e \in E} \in \bigoplus_{e \in E} H^1(0, \ell_e)
\end{align*}
which satisfy the additional condition that, for each vertex $v \in V$,
all functions $f_e$ on edges that begin or end in $v$ have the same value at $v$.
By $L^2(G)$, we simply mean the space $\bigoplus_{e \in E} L^2(0, \ell_e)$.
We define symmetric a bilinear form
\begin{align*}
	a: H^1(G) \times H^1(G) \to \bbC
\end{align*}
as
\begin{align*}
	a(f,g) = \sum_{e \in E} \int_0^{\ell_e} f_e'(x) \overline{g_e'(x)} \dx x
\end{align*}
for all $f,g \in H^1(G)$.
To the form $a$ we can associate a self-adjoint operator $-\Delta$ on $L^2(G)$;
its negative operator $\Delta$ has spectral bound $0$ and acts as the Laplace operator
on each edge of $G$;
its boundary conditions are, in addition to the continuity condition in the vertices
that has already appeared in the definition of $H^1(G)$ above,
so-called \emph{Kirchhoff conditions} in the vertices.

The operator $\Delta$ is a simultaneous generalization of the Neumann-Laplace operator 
and the Laplace operator with periodic boundary conditions on an interval:
if $G$ consists precisely of two vertices and one edge of length $\ell$ between them, 
then $\Delta$ is simply the Neumann-Laplace operator on $L^2(0,\ell)$.
If we connect two vertices by two edges, instead, then $\Delta$ becomes 
the Laplace operator with periodic boundary conditions on an interval.

The operator $\Delta$ generates a positive $C_0$-semigroup on $L^2(G)$
(see for instance \cite[Example~6.73]{Mugnolo2014})
and thus, the resolvent $\Res(\mu, \Delta)$ is a positive operator 
for every $\mu > \spb(\Delta) = 0$.
The following proposition shows 
that $\Delta$ also satisfies a uniform anti-maximum principle:

\begin{proposition}
	\label{prop:anti-max-laplace-on-graph}
	Consider the Laplace operator $\Delta$ that we introduced above on the graph $G$.
	For all $\mu$ in a left neighbourhood of $\spb(\Delta)$, we have $\Res(\mu,\Delta) \preceq -\one \otimes \one$.
	Here, $\one \in L^2(G)$ denotes the vector $(\one)_{e \in E}$.
\end{proposition}

\begin{proof}
	Since $\dom{\Delta}$ is a subset of $H^1(G)$, 
	the domination assumption in Setting~\ref{sett:general} is satisfied
	for $u = \varphi = \one$ and $m_1 = m_2 = 1$.
	Moreover, since the graph is assumed to be connected,
	the spectral bound $0$ is a simple eigenvalue of $\Delta$ 
	and its eigenspace is spanned by $\one$;
	thus, the spectral assumption in Setting~\ref{sett:general} is also satisfied.
	Furthermore, since $\Delta$ is a generator of a positive $C_0$-semigroup, all assumptions of Corollary~\ref{cor:anti-max-characterization} are satisfied. Thus, we only need to show $\Res(\mu,A) \preceq \one\otimes \one$ for some $\mu>0$.
	
	Since the form domain $H^1(G)$ of $a$ is contained in the principal ideal
	generated by $\one$ in $L^2(G)$,
	we can employ Proposition~\ref{prop:form-domain-estimates},
	which indeed yields the estimate $\Res(\mu,\Delta) \preceq \one \otimes \one$
	for all $\mu > 0$.
\end{proof}

The situation considered in the above example can be generalized in various directions,
but we refrain from discussing this in detail here.

\subsection{Polyharmonic operators with Dirichlet boundary conditions revisited}
\label{subsection:applications:polyharmonic-ball}

Let $\Omega$ be a bounded domain in $\bbR^n\, (n\geq 2)$, say with $C^\infty$-boundary. 
For a positive integer $\ell\geq 2$, consider the polyharmonic operator $A$ on $E:=L^2(\Omega)$ given by
\[
	A: 
	\dom{A} := W^{2\ell,2}(\Omega) \cap W_0^{\ell,2}(\Omega) \to E, 
	\quad 
	f \mapsto Af:=-(-\Delta)^\ell f.
\]
We define $d: \Omega\to\bbC$ as $d(x)=\dist(x,\partial \Omega)$.

It was shown by Grunau and Sweers in \cite[Theorem~5.2]{GrunauSweers1997b} 
that for domains sufficiently close to the unit ball, 
the leading eigenvalue $\spb(A)$ is geometrically simple 
and the corresponding eigenspace is spanned by a vector $v \succeq d^2$. 
This result helps us establish 
that the operator $A$ satisfies the spectral assumptions in Setting~\ref{sett:general}.
For more general domains, this property cannot be expected, in general (see the recent article \cite{SchniedersSweers2020b}, though) 
We collect here some properties of the operator $A$:

\begin{proposition}
	\label{prop:polyharmonic-ball-properties}
	The operator $A$ is a closed, densely defined, and real operator on $E$. 
	Moreover, $A$ is self-adjoint with a strictly negative spectral bound,
	and the domination assumption from Setting~\ref{sett:general} is satisfied
	for a sufficiently large number $m_1 = m_2$
	and for $u = \varphi = d^\ell$.
	
	If, in addition, $\Omega$ is sufficiently close to the unit ball in $\bbR^n$ 
	in the sense of \cite[Theorem~5.2]{GrunauSweers1997b}, 
	then $A$ also satisfies the spectral assumption in Setting~\ref{sett:general}
	for $\lambda_0 = \spb(A)$. 
\end{proposition}

\begin{proof}
	The operator $A$ is densely defined and real, and it's spectral bound satisfies $\spb(A)<0$. 
	Moreover, it follows by \cite[Corollary~2.21]{GazzolaGrunauSweers2010} that $0\in \rho(A)$. 
	Thus $A$ is closed and its self-adjointness follows with the aid of \cite[Theorem~8.4]{LionsMagenes1972}. 
	Hence all of its spectral values are real.
	
	In order to verify the domination assumption, note that for sufficiently large integers $m_1$ and $j$, 
	we have 
	\[
		\dom{A^{m_1}} \subseteq W^{2\ell m_1,2}(\Omega)\cap W_0^{\ell,2}(\Omega) \subseteq C^j\left(\overline{\Omega}\right) \cap W_0^{\ell,2}(\Omega) \subseteq E_u;
	\]
	where the first inclusion follows from \cite[Corollary~2.21]{GazzolaGrunauSweers2010},
	the second inclusion from a Sobolev embedding theorem, 
	and the third inclusion from the boundary decay rate of $u$.
	Combining the above facts with the self-adjointness of $A$, 
	it follows that the domination assumption in Setting~\ref{sett:general} is satisfied
	for sufficiently large values of $m_1$ and for $m_2 = m_1$.
	
	If $\Omega$ is sufficiently close to the unit ball then, 
	as stated above, we know from \cite[Theorem~5.2]{GrunauSweers1997b} 
	that the eigenspace $\ker(\spb(A)-A)$ is spanned by a vector $v\succeq u$,
	so the spectral assumption from Setting~\ref{sett:general} is satisfied as well. 
\end{proof}

For $\ell=2$, the operator $A$ generates a $C_0$-semigroup on $E$ 
and its eventual positivity properties were studied 
first in \cite[Section~6]{DanersGlueckKennedy2016b} and \cite[Sections~8.3 and 11.4]{Glueck2016} 
and then again in \cite[Section~4]{DanersGlueck2018b}. 
In fact, in the former references \emph{individual} eventual properties of the resolvent were also studied. 
In particular, an individual maximum principle was proved 
in \cite[Proposition~6.5]{DanersGlueckKennedy2016b} for $n<4$. 

In the following, we are interested in uniform maximum and anti-maximum principles
for this operator.
We recall two known results about this topic, 
show how these results fit into the general theory developed in this paper, and generalize the first of them.
Let us start with the following uniform maximum principle that was recently
proved by Schnieders and Sweers in \cite[Corollary~4]{SchniedersSweers2020} for the particular case $\ell=2$.

\begin{theorem}
	\label{thm:schnieders-sweers}
	Assume that $\lambda_0 \in \bbR$ is a geometrically simple eigenvalue of $A$
	with an associated eigenvector $v$ that satisfies $v \succeq u := d^2$
	(where $d(x) = \dist(x,\partial \Omega)$ as above).
	Then $\Res(\mu,A) \succeq u \otimes u$ for all $\mu$ in a right neighbourhood of $\lambda_0$.
\end{theorem}

As mentioned in Proposition~\ref{prop:polyharmonic-ball-properties},
the spectral assumption in the theorem is satisfied for $\lambda_0 = \spb(A)$
if $\Omega$ is sufficiently close to the unit ball.

The proof of Theorem~\ref{thm:schnieders-sweers} in the special case $\ell = 2$ 
given in \cite[Corollary~4]{SchniedersSweers2020} 
relies on an intriguing lower estimate for the Green's function of $A$ which is deduced,
via quite involved technical arguments, from the lower bound
$\Res(0,A) \succeq -u\otimes u$ that was already known in the literature 
(see \cite[Theorem~1.5]{DallAcquaMeisterSweers2005} and \cite[Theorem~1]{GrunauRobert2010}, or alternatively \cite[Theorem~4.1]{Pulst2015}).
Let us now demonstrate how Theorem~\ref{thm:schnieders-sweers} can be obtained 
as a direct consequence of this lower bound and our abstract main result, 
Theorem~\ref{thm:main}.

\begin{proof}[Proof of Theorem~\ref{thm:schnieders-sweers}]
	The domination assumption in Setting~\ref{sett:general} 
	is satisfied according to Proposition~\ref{prop:polyharmonic-ball-properties}.
	Moreover, the spectral assumption in Setting~\ref{sett:general}
	is satisfied due to the assumptions of the theorem.
	In addition, the estimate 
	\[
		\Res(0,A) \succeq -u\otimes u;
	\]
	follows from \cite[Theorem~4.1]{Pulst2015}.
	Taking $\mu_0=0$ in Theorem~\ref{thm:main}\ref{thm:main:item:max} 
	proves the assertion (since $\lambda_0 \le \spb(A) < 0)$.
\end{proof}

Very recently, Grunau showed a similar estimate as in \cite[Theorem~4.1]{Pulst2015} for a class of polyharmonic operators where the leading coefficients are non-constant \cite[Theorem~1]{Grunau2021}.

The second result that we discuss here is the following characterization 
of the uniform anti-maximum principle for polyharmonic operators on balls
which was announced in \cite{ClementSweers2001} and proved in \cite[Theorem~1]{GrunauSweers2001}
(though in a slightly different form than stated below).

\begin{theorem}
	\label{thm:polyharmonic-ball-anti-max}
	Fix an integer $k \ge 1$ and, as above, let the dimension $n$ be at least $2$.
	Assume, in addition, that the domain $\Omega \subseteq \bbR^n$ is a ball.  
	Then the following are equivalent for the operator $B:=-(-A)^k$:
	\begin{enumerate}[ref=(\roman*)]
		\item\label{item:polyharmonic-ball-anti-max-strong} 
		For all $\mu$ in a left neighbourhood of $\spb(B)$ we have $\Res(\mu,B) \preceq -u \otimes u$,
		where $u = d^\ell$.
		
		\item\label{item:polyharmonic-ball-negative-one-point} 
		There exists a number $\mu_0 < \spb(B)$ in the resolvent set of $B$ 
		at which we have $\Res(\mu_0,B) \leq 0$.
		
		\item\label{item:polyharmonic-ball-necessary} 
		The dimension $n$ satisfies $n<2\ell(k-1)$.
	\end{enumerate}
\end{theorem}

Strictly speaking, the equivalences in the above theorem 
are slightly stronger than the one in \cite[Theorem~1]{GrunauSweers2001},
since the uniform anti-maximum principle in \cite[Theorem~1(b)]{GrunauSweers2001}
is formally weaker than assertion~\ref{item:polyharmonic-ball-anti-max-strong}
in the theorem above,
but formally stronger than assertion~\ref{item:polyharmonic-ball-negative-one-point}.

More interesting, though, is the question how this characterization 
is proved in \cite[Theorem~1]{GrunauSweers2001}:
it is known from \cite[Theorem~1.2]{GrunauSweers2002} that
$\Res(0, B) \preceq u\otimes u$ if and only if $n<2\ell(k-1)$;
this is used to prove \cite[Theorem~1]{GrunauSweers2001} via various resolvent estimates.
Since we have encapsulated all the technical resolvent estimates into our abstract results,
we are now able to re-derive Theorem~\ref{thm:polyharmonic-ball-anti-max}
from \cite[Theorem~1.2]{GrunauSweers2002} via a very short argument:

\begin{proof}[Proof of Theorem~\ref{thm:polyharmonic-ball-anti-max}]
	Since $\Omega$ is a ball, $\Res(0,A)$ is positive;
	this was shown by Boggio in \cite{Boggio1905}.
	This means that $\Res(0, B)$ must also be positive.
	Moreover, according to Proposition~\ref{prop:polyharmonic-ball-properties},
	the operator $A$ satisfies the domination and spectral assumption 
	from Setting~\ref{sett:general}
	for $u = \varphi = d^\ell$ and for sufficiently large $m_1 = m_2$.
	Consequently, so does $B$.
	
	Finally, we note that $\Res(0, B) \preceq u\otimes u$ if and only if $n<2\ell(k-1)$ 
	by \cite[Theorem~1.2]{GrunauSweers2002}. 
	So the claimed equivalence is a consequence of Corollary~\ref{cor:anti-max-characterization}.
\end{proof}

We remark that Corollary~\ref{cor:anti-max-characterization} was obtained 
as a consequence of Theorem~\ref{thm:anti-max-characterization} which 
(as mentioned before) adapted the resolvent arguments 
of Grunau and Sweers in \cite[Theorem~1]{GrunauSweers2001} to a more abstract setting. 
Let us also stress that we could only apply Corollary~\ref{cor:anti-max-characterization} 
due to the explicit Green's function estimate 
which was given by Grunau and Sweers in \cite[Theorem~1.2]{GrunauSweers2002}.

\subsection{Differential operators of odd order}
\label{subsection:applications:odd-order}
	
After the various applications to elliptic type differential operators in the preceding subsections,
we will now consider a class of differential equations that are 
equations of odd order on an interval, with periodic boundary conditions.
Since the order is odd, the maximum and anti-maximum principle might come as quite a surprise.
Throughout this subsection, let $\ell \ge 0$ be an integer. 
We consider the operator
\begin{align*}
	A: 
	L^2(0,1) \supseteq \dom{A} & \to     L^2(0,1),  \\
	                         w & \mapsto w^{(2\ell+1)}
\end{align*}
with domain 
\begin{align*}
	\dom{A} := 
	\big\{
		w \in H^{2\ell+1}(0,1) 
		: \,
		w^{(k)}(0) = w^{(k)}(1) 
		\text{ for all } k=0,\dots,2\ell
	\big\}.
\end{align*}
The operator $A$ is skew-adjoint and 
thus generates a unitary $C_0$-group $(e^{tA})_{t \in \bbR}$ on $L^2(0,1)$.
In the simplest case $\ell = 0$, this group is simply the periodic shift group;
but we will see in Proposition~\ref{prop:odd-order-not-ev-pos} that the (semi)group does not satisfy any eventual positivity property
if $\ell \ge 1$.
More importantly though, we will show that -- for each choice of $\ell$ --
the operator $A$ satisfies a uniform maximum and a uniform anti-maximum principle 
at its only real spectral value $\lambda_0 = 0$.

Let us start by noting that our domination and spectral assumption are satisfied:

\begin{proposition}
	\label{prop:domination-and-spectral-condition-for-odd-order-operator}
	Let $u := \one \in L^2(0,1)$ denote the constant function which takes the value $1$.
	Moreover, under the canonical anti-linear identification of $L^2(0,1)$ with its dual space,
	we also set $\varphi := \one$.
	
	Then the domination and spectral assumptions in Setting~\ref{sett:general} are satisfied by $A$
	 for $m_1 = m_2 = 1$ and $\lambda_0 = 0$ respectively.
\end{proposition}

\begin{proof}
	The principal ideals $E_u$ and $(E')_\varphi$ that occur in the domination assumption
	are both equal to $L^\infty(0,1)$, due to our choice of $u$ and $\varphi$.
	
	Both $A$ and its adjoint $-A$ are defined on the space $H^{2\ell+1}(0,1)$;
	this is a subspace of $H^1(0,1)$ and the latter embeds into $L^\infty(0,1)$, 
	so the domination assumption is satisfied for $m_1 = m_2 = 1$, as claimed.
	
	Since $A$ is skew-adjoint, it suffices to prove the spectral condition for $A$ and $\one$ only.
	Obviously, $\one \in \ker A$. 
	On the other hand, every vector in $\ker A$ is a polynomial of degree at most $2\ell$,
	and the boundary conditions encoded in the domain of $A$ then readily imply 
	that this polynomial is actually constant. 
	Hence, $\ker A$ is indeed spanned by $\one$.
\end{proof}

Now we prove a uniform maximum and anti-maximum principle for the odd-order operator $A$.

\begin{theorem}
	\label{thm:odd-order-at-least-3-max-and-anti-max}
	Let $A$ be the differential operator of order $2\ell+1$ with periodic boundary conditions
	on $L^2(0,1)$ introduced above.
	\begin{enumerate}[\upshape (a)]
		\item 
		For all real numbers $\mu$ in a right neighbourhood of $0$,
		we have
		\[
			\Res(\mu,A) \succeq \one \otimes \one.
		\]
		
		\item 
		For all real numbers $\mu$ in a left neighbourhood of $0$,
		we have
		\[
			\Res(\mu,A) \preceq - \one \otimes \one.
		\]
	\end{enumerate}
	If $\ell = 0$ (i.e., if the operator $A$ is of first-order), 
	assertion~(a) even holds for all $\mu \in (0,\infty)$ and~(b) even holds for all $\mu \in (-\infty,0)$.	
\end{theorem}

\begin{proof}
	Recall from the previous proposition that $A$ satisfies
	both the domination and the spectral assumption in Setting~\ref{sett:general}
	for $u = \varphi = \one$.
	We use different arguments for the cases $\ell \ge 1$ and $\ell = 0$.
	
	\emph{First case: $\ell \ge 1$:}
	It suffices to prove the resolvent estimate
	\begin{align}
		\label{eq:thm:odd-order-at-least-3-max-and-anti-max-1}
		- \one \otimes \one \preceq \Res(\mu,A) \preceq \one \otimes \one
		\qquad \text{for all } \mu \in \bbR \setminus \{0\},
	\end{align}
	since both claims~(a) and~(b) then follow immediately from our main result,
	Theorem~\ref{thm:main}.
	In order to verify~\eqref{eq:thm:odd-order-at-least-3-max-and-anti-max-1}, 
	let us first show the estimate for $\mu = 1$.
	To this end, consider the closed operators 
	\begin{align*}
		B: w \mapsto w'
		\qquad \text{and} \qquad 
		C: w \mapsto \sum_{k=0}^{2\ell} w^{(k)}
	\end{align*}
	on $L^2(0,1)$ with domains
	\begin{align*}
		\dom{B} := 
		\big\{
			w \in H^1(0,1): \, w(0) = w(1)
		\big\}
	\end{align*}
	and
	\begin{align*}
		\dom{C} := 
		\big\{
			w \in H^{2\ell}(0,1) 
			: \,
			w^{(k)}(0) = w^{(k)}(1) 
			\text{ for all } k=0,\dots,2\ell-1
		\big\},
	\end{align*}
	respectively.
	We have $1-A = (1-B)C$ and we know that $1$ is in the resolvent set of $A$ and $B$.
	Thus, $0$ is in the resolvent set of $C$ with
	$\Res(1,A) = \Res(0,-C) \Res(1,B)$.
	The dual operator of $B$ has the same domain as $B$, 
	so the dual resolvent $\Res(0,B)'$ maps into $L^\infty(0,1) = L^2(0,1)_{\one}$;
	the resolvent $\Res(0,-C)$ maps into the same space as well, since since $\ell \ge 1$.
	
	Therefore, we obtain from Proposition~\ref{prop:domination-imlies-estimate-for-product} that, 
	indeed the required estimate
	$- \one \otimes \one \preceq \Res(1,A) \preceq \one \otimes \one$ holds.
	Now, Proposition~\ref{prop:extend-domination-two-sided} even shows that 
	the same estimate remains true for all real numbers $\mu$ in the resolvent set of $A$,
	i.e., we have proved~\eqref{eq:thm:odd-order-at-least-3-max-and-anti-max-1}
	and as a result the theorem for the case $\ell \ge 1$.
	
	\emph{Second case: $\ell = 0$.}
	In this case, we cannot use the previous argument since the operator $C$ from above
	would be the identity operator and thus, its resolvent does not map into $L^\infty(0,1)$.
	However, the case $\ell = 0$ is simple enough that one can just compute the resolvent explicitly.
	Using the periodic boundary conditions, one obtains the formula
	\begin{align*}
		\Res(\mu,A)f(x) 
		= 
		e^{\mu x}
		\left(
			\frac{e^\mu}{e^\mu-1} 
			\int_0^1 e^{-\mu y}f(y) \dx y
			-
			\int_0^x e^{-\mu y} f(y) \dx y
		\right)
	\end{align*}
	for all complex numbers $\mu$ in the resolvent set of $A$ 
	and all $f \in L^2(0,1)$.
	For non-zero real numbers $\mu$, this formula can be used to directly verify 
	the assertions of the theorem.
\end{proof}

As mentioned above, if $\ell = 0$, then the unitary group $(e^{tA})_{t \in \bbR}$
is just the periodic shift semigroup and hence, it is positive.
Let us show now that the $C_0$-semigroups $(e^{tA})_{t \in [0,\infty)}$
and $(e^{-tA})_{t \in [0,\infty)}$ are not \emph{individually eventually positive} for the case $\ell \ge 1$,
in the sense specified below:

\begin{proposition}
	\label{prop:odd-order-not-ev-pos}
	Let $\ell \ge 1$.
	Then there exists a function $0 \le f_1 \in L^2(0,1)$ such
	that the orbit $t \mapsto e^{tA}f_1$ is not eventually contained
	in the positive cone $L^2(0,1)_+$. 
	Likewise, there exists a function $0 \le f_2 \in L^2(0,1)$ such 
	that the orbit $t \mapsto e^{-tA}f_2$ is not eventually contained
	in $L^2(0,1)_+$.  
\end{proposition}

\begin{proof}
	By identifying $L^2(0,1)$ with the $L^2$-spaces over the complex unit circle
	and using the Fourier transform, we can see that the point spectra of $A$ and $-A$ are given by
	\begin{align*}
		\pntSpec(A) = \pntSpec(-A) = 
		\big\{
			2\pi i k^{2\ell + 1}
			: \,
			k \in \bbZ
		\big\}.
	\end{align*}
	Now assume that, for each $0 \le f_1 \in L^2(0,1)$, the orbit $t \mapsto e^{tA}f_1$
	was eventually contained in the positive cone.
	This property of the semigroup $(e^{tA})_{t \in [0,\infty)}$
	is called \emph{individual eventual positivity}, and for such semigroups
	it was shown in \cite[Theorem~2.1]{AroraGlueck2021}
	that intersection of the point spectrum $\pntSpec(A)$ with the imaginary axis is \emph{cyclic}, 
	which means that every integer multiple of a purely imaginary eigenvalue of $A$ is again an eigenvalue.
	But since $\ell \ge 1$, the set $\pntSpec(A) \cap i\bbR = \pntSpec(A)$ 
	-- which we computed above -- is obviously not cyclic.
	
	The same argument also applies to the semigroup generated by $-A$.
\end{proof}

The cyclicity result from \cite[Theorem~2.1]{AroraGlueck2021} even holds 
for a slightly larger class of semigroups, as shown in \cite[Theorem~6.3.2]{Glueck2016}.
If one uses the latter more general result, one can also generalize the previous proposition:
there exists a function $0 \le f_1 \in L^2(0,1)$ such that the orbit
$t \mapsto e^{tA}f_1$ does not converge to the positive cone as $t \to \infty$
(and likewise for $f_2$ and the semigroup in negative times).

\subsection{A delay differential operator}
\label{subsection:application-delay-operator}

In this last subsection, we consider a delay equation taken from \cite[Theorem~4.6]{DanersGlueck2018b},
where \emph{uniform} eventual positivity of the solution semigroup was asserted;
here we study the eventual positivity of the corresponding stationary problem, i.e., 
of the resolvent of the differential operator.

Fix a real number $c \in (0,\infty)$, 
and consider the time evolution of a function $y:[-2,\infty)\to\bbC$ given by
\begin{equation}
	\label{eq:delay}
	\dot{y}(t) 
	= 
	c\left( \int_{t-2}^{t-1} y(s)\, ds -\int_{t-1}^t y(s) \dx s + y(t-2)-y(t)\right) 
	\quad 
	\text{ for } t \geq 0.
\end{equation}
Let $E = \bbC \times L^1\left([-2,0]; \bbC\right)$ and define a 
closed, densely defined linear operator $A:\dom{A}\subseteq E\to E$ by
\begin{align*}
	\dom{A} & := \left\{ (x,f) \in \bbC\times W^{1,1}\left( (-2,0);\bbC\right) : f(0)=x \right\},
	\\
	A(x,f)  & := \left( \duality{\Phi}{f} , f' \right);
\end{align*}
where $\Phi$ is a functional on $W^{1,1}\left( (-2,0);\bbC\right)$ given by
\[
	\duality{\Phi}{f} 
	= 
	c\left( \int_{-2}^{-1} f(s)\, ds -\int_{-1}^0 f(s) \dx s + f(-2)-f(0) \right)
\]
for all $f\in W^{1,1}\left( (-2,0);\bbC\right)$. 
Then $A$ generates a $C_0$-semigroup by \cite[Theorem~3.23]{BatkaiPiazzera2005} 
and the solutions of the delay differential equation \eqref{eq:delay} 
on the \emph{state space} $E$ are given by this semigroup 
(see \cite[Section~3.1]{BatkaiPiazzera2005}). 
We point out that \cite{BatkaiPiazzera2005} only considers the interval $[-1,0]$ 
but the results still hold if we stretch the interval. 
This semigroup was also considered in \cite[Theorem~4.6]{DanersGlueck2018b}, 
where uniform eventual positivity properties were proved for $c\leq \pi/16$. 
Here, we instead treat the resolvent. 
In fact, we show that both the uniform maximum and anti-maximum principles 
are satisfied without any restriction on $c$.

Before stating the result, we point out that the dual space $E'$ can be identified 
with the space $\bbC \times L^{\infty}\left([-2,0]; \bbC\right)$.
We also note that $0$ is an eigenvalue of $A$ with eigenvector $(1,\one_{[-2,0]})$.

\begin{theorem}
	Let $A$ be the linear operator given above on $E=\bbC \times L^1\left([-2,0]; \bbC\right)$. 
	Let $u:= (1,\one_{[-2,0]}) \in E$ and $\varphi =(1,c\one) \in E'$.
	\begin{enumerate}
		\item 
		For all $\mu$ in a right neighbourhood of $0$, we have
		\[
			\Res(\mu,A) \succeq u\otimes \varphi.
		\]
		\item 
		For all $\mu$ in a left neighbourhood of $0$, we have
		\[
			\Res(\mu,A) \preceq -u\otimes \varphi.
		\]
	\end{enumerate}
\end{theorem}

\begin{proof}
	We begin by noting that the operator $A$ is real. 
	Next, by the Sobolev embedding theorem, 
	$\dom{A} \subseteq \bbC \times L^{\infty}\left([-2,0]; \bbC\right) = E_u$. 
	Moreover, $\varphi$ is a strictly positive functional and $(E')_{\varphi} = E'$. 
	Therefore the domination assumption in Setting~\ref{sett:general} is satisfied for $m_1=1$ and $m_2=0$.
	
	As mentioned above, a simple computation shows that $0$ is an eigenvalue of $A$ and the eigenspace is spanned by the vector $u$. 
	In particular, $0$ is a geometrically simple eigenvalue of both $A$ and $A'$. 
	Define $v(s) =3+s$ if $s\in [-2,-1]$ and $v(s) =1-s$ for $s\in (-1,0]$. 
	It was shown in the proof of \cite[Theorem~4.6]{DanersGlueck2018b} that $\psi:=(1,cv)\in \ker A'$. 
	Of course, $\psi \succeq \varphi$. 
	This settles the spectral assumption in Setting~\ref{sett:general} as well. 
	
	Finally, as $m=m_1+m_2=1$, an appeal to Theorem~\ref{thm:ev-pos-powers} concludes the proof.
\end{proof}

\subsection*{Acknowledgements} 

We are grateful to Delio Mugnolo for his suggestion 
to study eventual positivity properties of the third-order equation 
in Subsection~\ref{subsection:applications:odd-order} (for $\ell = 1$).
The authors are also grateful to the referee for their helpful comments and, in particular, for pointing out the references \cite{Grunau2021} and \cite{Pulst2015}.

The first author would like to express their thanks 
for the hospitality during a very pleasant stay at Universit\"at Passau, where most of the work was done. 
The first named author was supported by 
Deutscher Aka\-de\-mi\-scher Aus\-tausch\-dienst (Forschungs\-stipendium-Promotion in Deutschland).

\bibliographystyle{plain}
\bibliography{literature}

\end{document}